\documentclass{compositio}

\usepackage{amsmath,amsfonts,yhmath,hyperref}
\usepackage{amsthm}
\usepackage{amssymb}
\usepackage[latin9]{inputenc}
\usepackage[nottoc,notlof,notlot]{tocbibind}
\usepackage{mathtools}
\usepackage{enumitem}
\usepackage{tikz-cd}
\usepackage{multirow}
\usepackage{comment}

\usepackage{float}
\usepackage{array}
\usepackage{mathrsfs}

\usepackage{pgf,tikz}

\usetikzlibrary{arrows}
\usetikzlibrary[patterns]
\usetikzlibrary{shapes}
\definecolor{qqqqff}{rgb}{0.,0.,1.}
\definecolor{cqcqcq}{rgb}{0.7529411764705882,0.7529411764705882,0.7529411764705882}
\definecolor{ttqqqq}{rgb}{0.2,0.,0.}
\definecolor{qqqqff}{rgb}{0.,0.,1.}
\definecolor{xdxdff}{rgb}{0.49019607843137253,0.49019607843137253,1.}
\definecolor{zzttqq}{rgb}{0.6,0.2,0.}
\definecolor{cqcqcq}{rgb}{0.7529411764705882,0.7529411764705882,0.7529411764705882}

\definecolor{yqyqyq}{rgb}{0.5019607843137255,0.5019607843137255,0.5019607843137255}
\definecolor{uuuuuu}{rgb}{0.26666666666666666,0.26666666666666666,0.26666666666666666}
\definecolor{xdxdff}{rgb}{0.49019607843137253,0.49019607843137253,1.}
\definecolor{qqqqff}{rgb}{0.,0.,1.}

 \font\ncsc=cmcsc10

\newcommand{\PP}{\mathbb{P}}
\newcommand{\NN}{\mathbb{N}}
\newcommand{\ZZ}{\mathbb{Z}}
\newcommand{\RR}{\mathbb{R}}
\newcommand{\CC}{\mathbb{C}}

\newcommand{\QQ}{\mathbb{Q}}

\newcommand{\TT}{\mathbb{T}}

\newcommand{\AAA}{\mathscr{A}}
\newcommand{\BBB}{\mathscr{B}}
\newcommand{\CCC}{\mathscr{C}}
\newcommand{\EEE}{\mathscr{E}}
\newcommand{\MMM}{\mathscr{M}}

\newcommand{\Dfk}{\mathfrak{D}}
\newcommand{\mfk}{\mathfrak{m}}
\newcommand{\hfk}{\mathfrak{h}}
\newcommand{\tfk}{\mathfrak{t}}

\newcommand{\E}{\mathcal{E}}
\renewcommand{\P}{\mathcal{P}}
\renewcommand{\L}{\mathcal{L}}
\newcommand{\N}{\mathcal{N}}
\newcommand{\M}{\mathcal{M}}

\newcommand{\G}{\mathcal{G}}
\newcommand{\J}{\mathcal{J}}

\newcommand{\Q}{\mathcal{Q}}

\renewcommand{\O}{\mathcal{O}}
\newcommand{\F}{\mathcal{F}}

\newcommand{\gen}[1]{\langle #1 \rangle}

\newcommand{\val}{\mathrm{val}}

\newcommand{\quant}[1]{\left[ #1 \right]}

\newtheorem{theo}{Theorem}[section]
\newtheorem*{theom}{Theorem}
\newtheorem{prop}[theo]{Proposition}
\newtheorem{coro}[theo]{Corollary}
\newtheorem{lem}[theo]{Lemma}

\theoremstyle{definition}
\newtheorem{defi}[theo]{Definition}

\theoremstyle{remark}
\newtheorem{remark}[theo]{Remark}
\newenvironment{rem}[1]{
    \begin{remark}#1}{
    \xqed{\blacklozenge}\end{remark}
}

\theoremstyle{remark}
\newtheorem{example}[theo]{Example}
\newenvironment{expl}[1]{
    \begin{example}#1}{
    \xqed{\lozenge}\end{example}
}

\newcommand{\xqed}[1]{
    \leavevmode\unskip\penalty9999 \hbox{}\nobreak\hfill
    \quad\hbox{\ensuremath{#1}}}

\classification{14N10, 14T90, 05A15, 14H99}
\keywords{Enumerative geometry, tropical refined invariants, relative invariants, floor diagrams\\ 
Thomas Blomme,
{\ncsc Universit\'e de Gen\`eve, 5-7 rue du Conseil G\'en\'eral,
1205 Gen\`eve,
Switzerland} \\
\textit{Email :} thomas.blomme@unige.ch\\
 Victoria Schleis, {\ncsc Universit\"at T\"ubingen, Auf der Morgenstelle 10,
72076 T\"ubingen,
Germany} \\
\textit{Email :} victoria.schleis@student.uni-tuebingen.de}
 
\begin{document}
 
 
\title{Tropical M\"obius strips and ruled surfaces}
\author{Thomas Blomme, Victoria Schleis}

\begin{abstract}
We consider the enumeration of tropical curves in M\"obius strips for two different lattice structures and relate them to the enumeration of curves in two rational ruled surfaces over a complex elliptic curve. Using this correspondence, we prove regularity results such as the piecewise quasi-polynomiality of relative invariants and the quasi-modularity of their generating series.
\end{abstract}

\maketitle

\tableofcontents

\section{Introduction}

\subsection{Setting}

	\subsubsection{Ruled surfaces.} Ruled surfaces are a wide family of algebraic surfaces. They are defined as $\CC P^1$-bundles over a base curve which possess some section. One way to obtain such a surface is to consider the projective completion $\PP(\L\oplus\O)$ of a line bundle $\L$ over the base curve, where $\O$ denotes the trivial line bundle. If the base curve is of genus $0$, these are the only ruled surfaces, known as \textit{Hirzebruch surfaces}. As complex surfaces, they only depend on the absolute value of the degree of $\L$. If the base curve is of genus $1$, there are two additional ruled surfaces which are not obtained as a projective completion. See \cite{hartshorne2013algebraic} and Section \ref{sec-complex-setting} for more details. The classification for base curves of higher genus is more complicated.

	The enumerative (and relative) invariants of ruled surfaces over an elliptic curve obtained as the projective completion of a line bundle have already been studied using tropical methods in \cite{blomme2021floor}. We refer to \cite{brugalle2014bit} for a short introduction to tropical geometry and to \cite{gathmann2007numbers} for general tropical enumerative problems. Informally, the tropical world is a combinatorial counterpart to the complex one. \textit{Tropicalization} is a process to get tropical objects out of (families of) complex objects. Correspondence statements relate enumeration and invariants in both worlds. For more details, see Section \ref{sec-complex-setting}. In \cite{blomme2021floor} tropical curves on cylinders are studied, and \textit{tropical cylinders} are obtained as the quotient of $\RR^2$ by the $\ZZ$-action generated by
	$$\psi_{l,\delta,\alpha}:(x,y)\longmapsto (x+l,y+\delta x+\alpha),$$
	for some choice of $l,\alpha\in\RR$ and $\delta\in\ZZ$. Thus, they have a natural lattice structure with monodromy around the cylinder and $\RR$-bundles over $\RR/l\ZZ$. Tropical curves in the cylinder are graphs on the cylinder such that edges have integer slope and satisfy a balancing condition. This is a tropical counterpart to the construction of line bundles over an abelian variety (here an elliptic curve) as done in \cite{griffiths2014principles}. Now, we consider other free $\ZZ$-actions on $\RR^2$. Their quotient space is not a cylinder anymore, but a M\"obius strip. There are two lattice structures on a M\"obius strip obtainable this way, see Section \ref{sec-tropical-mobius-strips} for a more detailed construction. 
	\begin{itemize}[label=$\ast$]
	\item The first M\"obius strip $\TT M_0$ is obtained as quotient of $\RR^2$ by
	$$\varphi_0:(x,y)\longmapsto (x+l,-y).$$
	It is a tropicalization of a (family of) $2$-torsion line bundle over an elliptic curve. The difference to the tropicalization presented in \cite{blomme2021floor} is that the divisor sitting at infinity is some multisection of size $2$ rather than the two natural sections of the projective completion of the $2$-torsion line bundle. We denote its complex counterpart by $\CC M_0$.
	\item The second M\"obius strip $\TT M_1$ is obtained as the quotient of $\RR^2$ by
	$$\varphi_1:(x,y)\longmapsto(x+l,x-y).$$
	It is a tropicalization of the ruled surface $\PP(\mathscr{E})$, where $\mathscr{E}$ is a non-split plane bundle over an elliptic curve that fits the following short exact sequence:
	$$0\to\O\to\mathscr{E}\to \O(p)\to 0,$$
	see \cite{hartshorne2013algebraic} for more details on this ruled surface, which we denote by $\CC M_1$.
	\end{itemize}
	Perhaps surprisingly, these two M\"obius strips are not both \textit{tropicalizations} of the two additional ruled surfaces not obtained as the projective completion of a line bundle.

	\subsubsection{Enumerative problems.} 
	Our goal is to study the enumerative invariants of $\CC M_0$ and $\CC M_1$ relative to the boundary divisor using tropical methods. Consider the surface $\CC M_\delta$ (for $\delta=0$ or $1$), a homology class $\beta\in H_2(\CC M_\delta,\ZZ)$, (later determined by two half-integers $(a,b)$), and a positive number $g$. We wish to compute the number of curves of genus $g$ in homology class $\beta$ passing through $-K\cdot\beta+g-1$ points in general position, where $K$ is the canonical class. The choice of $(a,b)$ ensures that $-K\cdot\beta=2b$. A more modern way of defining these numbers is to consider \textit{log-Gromov-Witten invariants}, obtained by integrating some cohomology classes over a \textit{virtual fundamental class} on the moduli space of log-stable maps. This technology was introduced in \cite{abramovich2014stable,chen2014stable,gross2013logarithmic}.  See Section \ref{sec-complex-setting} for more details on these invariants.
	
	\smallskip
	
	Each surface $\CC M_\delta$ is endowed with the choice of a divisor $D$ which is a $2$-section, and numerically equivalent to the canonical class. The intersection number $D\cdot\beta$ is equal to $2b$. Thus, we can choose two partitions $\mu$ and $\nu$ such that $\|\mu\|+\|\nu\|=2b$. If $\mu$ is a partition, we set $\|\mu\|=\sum i\mu_i$ and $|\mu|=\sum\mu_i$. We then count curves of genus $g$ of class $\beta$ that have $\mu_i$ tangencies of order $i$ at fixed points on $D$, $\nu_i$ additional non-fixed tangencies of order $i$, and pass through $|\nu|+g-1$ points. The number of such curves can also be expressed in the setting of log-GW invariants, and is denoted by $\N^\delta_{g,aE+bF}(\mu,\nu)$ and called \emph{relative invariant}. If $\mu=\emptyset$ and $\nu=1^{2b}$, we recover the non-relative invariant.
	
	\smallskip
	
	Similar invariants have already been considered for toric surfaces, in particular Hirzebruch surfaces, see \cite{ardila2017double}. It is possible to study the count of curves of fixed genus in a fixed homology class passing through a suitable number of points using tropical methods and Mikhalkin's correspondence theorem \cite{mikhalkin2005enumerative}. The latter has been a groundbreaking result and lead to many generalizations and different proofs since, see \cite{nishinou2020realization, nishinou2006toric, shustin2002patchworking,shustin2004tropical,tyomkin2017enumeration}. Moreover, in the toric setting, the correspondence theorem allows the computation of  invariants relative to the toric boundary, which are obtained by counting curves which have a prescribed tangency profile with the boundary divisors.

	\smallskip
	
	The use of log-geometry techniques by T.~Nishinou and B.~Siebert \cite{nishinou2006toric} lead to the introduction of log-GW invariants \cite{abramovich2014stable,chen2014stable,gross2013logarithmic}. This setting allows to consider invariants that generalize usual GW invariants, and that can be defined for reducible surfaces. It can be shown that these invariants are constant in families, as done by T. Mandel and H. Ruddat in \cite[Appendix A]{mandel2020descendant} (also done in \cite{abramovich2020decomposition}). In some situations, the invariants of a reducible surface can be expressed in terms of the invariants for the components, leading to degeneration formulas and a correspondence between tropical curves and enumerative invariants. See for instance \cite{kim2018degeneration} for a proof of J. Li degeneration formula \cite{li2002degeneration} in the log-GW setting.
	
	\smallskip
	
	The correspondence theorem translates the algebraic enumerative problem into a combinatorial problem, which still needs to be solved. Several tools have been developed to tackle these tropical enumerative problems, including lattice path algorithms \cite{mikhalkin2005enumerative}, recursive formulas \cite{blomme2019caporaso,gathmann2007caporaso,gathmann2008kontsevich}, and floor diagrams algorithms \cite{blomme2021floor,blomme2022abelian3, brugalle2008floor}. These techniques can then be used to prove regularity results on the corresponding invariants. For instance, quasi-modularity of generating series \cite{blomme2021floor, boehm2018tropical}, piecewise polynomiality of relative invariants \cite{ardila2017double}, or the  polynomiality of coefficients of certain \textit{refined invariants} \cite{brugalle2020polynomiality}.

\subsection{Results}

In this paper, we apply tropical correspondence techniques to compute enumerative and relative invariants of the ruled surfaces $\CC M_0$ and $\CC M_1$ by studying tropical curves on tropical M\"obius strips. Due to the non-orientability of M\"obius strips, several new features appear when studying tropical curves. We then give a floor diagram algorithm enabling concrete computations and use the latter to prove several regularity statements.

	\subsubsection{Correspondence statement.} Using the decomposition formula from \cite{abramovich2020decomposition}, we give a correspondence statement in Theorem \ref{theorem-corresp} that relates the log-GW invariants of the considered ruled surfaces to the tropical count, i.e. we show $\N^\delta_{g,aE+bF}(\mu,\nu) = N^\delta_{g,aE+bF}(\mu,\nu)$, where the left-hand side denotes the log-GW invariant, and the right-hand side the tropical enumerative invariant. The proof proceeds by constructing a family of surfaces with a central fiber to which we can apply the decomposition formula. The proof follows the same steps and can be seen as a particular case of \cite{cavalieri2021counting} for Hirzebruch surfaces (but without $\psi$-classes), and \cite{bousseau2019tropical} (but without $\lambda$-classes). We refer to these references for more details on log-GW invariants and other applications of the decomposition formula.

	More precisely, the correspondence theorem assigns to each tropical curve a multiplicity. The latter is constructed as follows. We consider tropical curves of fixed genus and degree. We refer to Section \ref{sec-tropical-curves} for definitions concerning tropical curves. We fix a configuration of points $\P$, matching the complex enumerative problem. Consider a solution $h:\Gamma\to\TT M_\delta$ to the tropical problem. We show that the connected components of $\Gamma-h^{-1}(\P)$, i.e. the complement of point conditions, are of one of the following two types:
	\begin{enumerate}[label=(\roman*)]
	\item it contains a unique end and no cycles,
	\item it contains a unique disorienting cycle (realizing an odd homology class in the M\"obius strip) and no end.
	\end{enumerate}
	In the case of tropical curves in $\RR^2$ \cite{mikhalkin2005enumerative} or cylinders \cite{blomme2021floor}, only components of the first type can appear: cycles in the complement of marked points are orienting and yield a 1-parameter family of solutions. As the M\"obius strip is not orientable, components of the second kind  arise due to disorienting cycles, which cannot be deformed. The multiplicity is given by
	$$m_\Gamma=2^k\prod_V m_V,$$
	where the product is over the trivalent vertices of the curve and $m_V$ is the usual vertex multiplicity, i.e. the absolute value of the determinant of two out of the three outgoing slopes. The exponent $k$ is the number of connected components of type (ii).

	\subsubsection{Refined invariants.} The tropical multiplicity contains the product $\prod m_V$. It was suggested in \cite{block2016refined} to replace the vertex multiplicity $m_V$ by its quantum analog $[m_V]=\frac{q^{m_V/2}-q^{-m_V/2}}{q^{1/2}-q^{-1/2}}$, which is a Laurent polynomial in the variable $q$. In the toric setting, it was proven that this multiplicity leads to a tropical invariant, called \textit{refined} invariant in \cite{itenberg2013block}. This refinement was conjectured in \cite{gottsche2014refined} to coincide with a refinement of the Euler characteristic by the Hirzebruch genus of some relative Hilbert scheme. Since, these tropical refined invariants have instead been proven to be related to Gromov-Witten invariants with $\lambda$-classes insertions, see \cite{bousseau2019tropical}, and to a refined count of real curves in the genus $0$ case, see \cite{mikhalkin2017quantum}.
	
	Tropical refined invariants have been generalized both in the toric setting by varying the type of conditions \cite{gottsche2019refined,shustin2018refined}, and out of the toric surface setting in various situations \cite{blomme2021floor,blomme2021refinedtrop,blomme2022abelian1,blomme2022abelian2}. In the toric surface setting, refined invariants have been proven to satisfy some regularity assumptions: for fixed genus, the coefficients of fixed codegree (seen as a function in the degree) are ultimately polynomial functions, see \cite{brugalle2020polynomiality}. Moreover, the authors of \cite{brugalle2020polynomiality} conjecture that the generating series of these polynomials depends on the surface in a universal way. This can be interpreted as a dual G\"ottsche conjecture \cite{kool2011short}. Expanding the number of cases where it is possible to study the refined invariants helps understanding this conjecture outside of the toric setting. Adapting the proof of the correspondence from \cite{bousseau2019tropical} would yield a generalization of \cite[Theorem 1]{bousseau2019tropical} and relate the refined invariants to the generating series of log-GW invariants with a $\lambda$-class insertion. See Remark \ref{rem-lambda-class}. To avoid these technicalities, we prove in Theorem \ref{theorem-refined-invariance} the existence of tropical refined invariants in the M\"obius strip case by studying the tropical walls.

	\subsubsection{Floor diagram algorithm.} To compute and study the tropical enumerative invariants, one needs a concrete way of solving the tropical enumerative problem. In the toric setting for $h$-transverse polygons, such an algorithm is provided by \textit{floor diagrams}, introduced in \cite{brugalle2008floor}. The idea is to stretch the point constraints in the vertical direction, so that curves become \textit{floor-decomposed}. The enumeration of solutions thus become an enumeration of $1$-dimensional graphs called \textit{diagrams}, with some multiplicity. The techniques were later developed for different settings \cite{blomme2021floor,brugalle2015floor}. Floor decomposition can also be interpreted on the algebraic side using Li's degeneration formula \cite{li2002degeneration} by using degeneration of a surface to the normal bundle of some curve. See for instance \cite{brugalle2015floor} for such an application.

	In Section \ref{sec-floor-diag-algo}, we provide a floor diagram algorithm suited for tropical curves on a M\"obius strip. To some extent, it is similar to floor diagrams on a cylinder presented in \cite{blomme2021floor}, as the boundary of the M\"obius strip is also an elliptic curve. The difference lies in the fact that the M\"obius strip only has one boundary component. The floor diagrams are thus infinite only on one side, the other side acquires \textit{ground floors}, i.e. floors with a disorienting cycle, and \textit{joints}, i.e. edges passing through the center of the M\"obius strip. See Section \ref{sec-floor-diag-algo} for details.

	\subsubsection{Quasi-modularity of generating series.} We now get to the two main results of the paper, which are about the regularity of the generating series of the enumerative invariants. The first one concerns the generating series
	$$F^\delta_{g,b}(\mu,\nu)(y)=\sum_{\substack{a\in\frac{1}{2}\NN \\ 2\delta a\equiv 2b \text{ mod }2} }N^\delta_{g,aE+bF}(\mu,\nu)y^{2a}.$$

	\begin{theom}\ref{theorem-quasi-modularity}
	The generating series $F^\delta_{g,b}(\mu,\nu)$ are quasi-modular forms for some finite index subgroup of $SL_2(\ZZ)$.
	\end{theom}
	
	The difference to the corresponding statement in \cite{blomme2021floor} is that here, it is quasi-modular only for a \emph{finite index subgroup} of $SL_2(\ZZ)$. This is due to the shape of the diagrams, whose multiplicity can not be expressed polynomially in terms of the Eisenstein series $G_2(y)=\sum_{a=1}^\infty \sigma_1(a)y^a$ and its derivatives.
 
	\subsubsection{Piecewise (quasi-)polynomiality of relative invariants.} The second main result deals with the relative invariants. Regularity of relative invariants was previously studied for double Hurwitz numbers \cite{goulden2005towards} (a $1$-dimensional analog of our problem), Hirzebruch surfaces \cite{ardila2017double}, and line bundles over an elliptic curve \cite{blomme2021floor}. The regularity concerns the relative invariants where we fix the lengths $|\mu|,|\nu|$ of the partitions $\mu$ and $\nu$, and we consider their entries as variables. In \cite{ardila2017double, blomme2021floor, goulden2005towards}, it is shown that the relative invariants are piecewise polynomials in the entries of the partitions, using the existence of floor decomposition for the tropical enumerative problems.
	
	Using the floor diagrams introduced in Section \ref{sec-floor-diag-algo}, we obtain \emph{piecewise quasi-polynomiality} of the invariants. We fix the information of $g$, $a$, as well as the lengths of the partitions $|\mu|$ and $|\nu|$. Then, we consider the invariant $N^\delta_{g,aE+bF}(\mu,\nu)$ (with $2b=\|\mu\|+\|\nu\|$) as a function in the entries of the partition. we have the following.
	
	\begin{theom}\ref{theorem-quasi-polynomiality}
	As a function in the entries of $\mu$ and $\nu$, $N^\delta_{g,aE+bF}(\mu,\nu)$ is piecewise quasi-polynomial.
	\end{theom}
	
	The latter means that there is some polyhedral subdivision such that on each cell, the function is given by the restriction of a \textit{quasi-polynomial}. A quasi-polynomial on $\ZZ^N$ is a function $f$ such that there exists a finite index sublattice $\Lambda$ and $f$ is polynomial on each equivalence $z+\Lambda$ of $\ZZ^N/\Lambda$. 
 
    We only get quasi-polynomiality due to the congruence conditions that appear in the description of the floor diagrams. It is possible to give an explicit description of the lattice $\Lambda$, and the quotient by $\Lambda$ is in fact a group whose elements are  $2$-torsion. Using this explicit description, we are able to prove piecewise polynomiality in several cases: if there is a unique point of maximal tangency contact (Corollary \ref{coro-one-end}), or if the genus $g$ is at most $2$ (Corollary \ref{coro-genus12}). If the genus is $3$, we show that the invariant is not piecewise polynomial (Example \ref{expl-diag-with-non-poly-mult} and Figure \ref{fig-floor-diag-nonpoly-mult}). 

    It would be interesting to study more relative invariants to see if piecewise quasi-polynmiality is a general feature.

\subsection{Structure}

The paper is organized as follows. We start by defining tropical M\"obius strips and the tropical curves they contain. We then consider the above mentioned enumerative problem and prove the existence of tropical invariants. The tropical invariants are then related to log-GW invariants defined in Section \ref{sec-complex-setting} via an application of the degeneration formula. The fifth section is devoted to the construction of the floor diagram algorithm. Floor diagrams are then used to prove the regularity statements for the corresponding invariants.

\smallskip

{\it Acknowledgements.}
We thank Hannah Markwig for feedback on an earlier draft and for helpful conversations.
Parts of this project were developed during the MSRI Summer School \emph{Tropical Geometry} in August 2022, and during a visit of T.B. to T\"ubingen in Fall 2022.
V.S. was supported by the Deutsche Forschungsgemeinschaft (DFG, German Research
Foundation), Project-ID 286237555, TRR 195.

\section{Tropical curves on M\"obius strips}\label{sec-tropical-mobius-strips}

\subsection{Tropical structures on M\"obius strips}

Topologically, a M\"obius strip is obtained from $[0;l]\times\RR$ by gluing together the two boundary components, reversing the orientation. The quotient obtained this way is a non-orientable surface. To endow it with a lattice structure, we regard it as a quotient of $\RR^2$ by the $\ZZ$-action generated by a fixed-point free orientation-reversing diffeomorphism.

\begin{prop}\label{prop: definition moebius strip}
There are two M\"obius strips (up to the choice of a positive real number $l$) obtained as the quotient of $\RR^2$ by a $\ZZ$-action:
\begin{itemize}[label=$\ast$]
\item the M\"obius strip $\TT M_0$, obtained as the quotient of $\RR^2$ by the action of
	$$\varphi_0:(x,y)\longmapsto (x+l,-y), \text{ and}$$
	\item the M\"obius strip $\TT M_1$, obtained as the quotient of $\RR^2$ by the action of
	$$\varphi_1:(x,y)\longmapsto (x+l,-y+x).$$
\end{itemize}
\end{prop}

\begin{proof}
To induce a lattice structure on the quotient using the natural lattice structure of $\RR^2$, we  consider lattice preserving diffeomorphisms $\varphi:\RR^2\mapsto 
\RR^2$, i.e. whose derivative lies in $GL_2(\ZZ)$.Up to a change of coordinates, this forces $\varphi$ to be affine. As we additionally require it to be fixed point free, it is of the form
$$\varphi_{l,\pm,\delta,\alpha}:(x,y)\longmapsto(x+l,\pm y +\delta x+\alpha)$$
for some $l>0$, $\delta\in\ZZ$, choice of sign and $\alpha\in\RR$. The choice of $+y$ induces orientation preserving diffeomorphisms and leads to the case of cylinders considered in \cite{blomme2021floor}. We are left with the choice of $-y$, giving the possible lattice structures on M\"obius strips.

\smallskip

The conjugation by an invertible affine map preserving the vertical direction gives
$$\begin{pmatrix}
1 & 0 \\
a & 1 \\
\end{pmatrix}\begin{pmatrix}
1 & 0 \\
\delta & -1 \\
\end{pmatrix}\begin{pmatrix}
1 & 0 \\
-a & 1 \\
\end{pmatrix} = \begin{pmatrix}
1 & 0 \\
\delta+2a & -1 \\
\end{pmatrix}.$$
Thus, for non-orientable diffeomorphisms, only the value of $\delta$ mod $2$ matters. Up to a change of coordinates given by the translation $y=\tilde{y}+\alpha/2$, $\alpha$ can then be assumed to be $0$. Thus, we have two families of tropical M\"obius strips that only differ by a scaling factor inside each family.
\end{proof}

Concretely, both M\"obius strips are obtained  by gluing the two boundary components of $[0;l]\times\RR$ via $(0,y)\sim (l,-y)$. The strip $[0;l]\times\RR$  is a fundamental domain for the action of $\varphi_\delta$ on $\RR^2$. However, the monodromy of the lattice structure under the quotient differs: it is $\left(\begin{smallmatrix} 1 & 0 \\ 0 & -1 \\ \end{smallmatrix}\right)$ for $\TT M_0$   and  $\left(\begin{smallmatrix} 1 & 0 \\ 1 & -1 \\ \end{smallmatrix}\right)$ for $\TT M_1$. For instance, a curve with horizontal slope crossing the right boundary comes back from the left boundary with slope $0$ for $\TT M_0$ but with slope $1$ for $\TT M_1$, compare Figures \ref{figure curves TM0} and \ref{figure curves TM1}.

The projection onto the first coordinate makes both M\"obius strips fiber over the tropical elliptic curve $\TT E=\RR/l\ZZ$. Adding the points at top and bottom infinity, i.e. considering the action extended to $\RR\times[-\infty;+\infty]$, they can thus be seen as $\TT P^1$-bundles over a tropical elliptic curve, just as the cylinders from \cite{blomme2021floor}. It is natural to expect that they arise as the tropicalization of $\CC P^1$-bundles over an elliptic curve. We refer to Section \ref{sec-complex-setting} for this matter.

Both strips have two to one covers by tropical cylinders: $\TT M_0$ is covered by $\TT C_0=\RR^2/\langle\varphi_0^2\rangle$, which is the total space of the trivial line bundle over $\TT\widetilde{E}=\RR/2l\ZZ$; and $\TT M_1$ is covered by  $\TT C_1$, the the total space of the unique $2$-torsion tropical line bundle over $\TT\widetilde{E}$.

\subsection{Parametrized tropical curves}\label{sec-tropical-curves}

\subsubsection{Abstract tropical curves.} An \textit{abstract tropical curve} is a metric graph with unbounded edges called \textit{ends} that have infinite length and are only adjacent to a unique vertex. The number of edges adjacent to a vertex of $\Gamma$ is called its \textit{valence}. The edges that are not ends are required to have finite length and are called \textit{bounded edges}. We denote the set of vertices by $V(\Gamma)$, the set of edges by $E(\Gamma)$ and the set of all bounded edges by $E_b(\Gamma)$. The \textit{genus} of an abstract tropical curve is defined by $g=1-\chi(\Gamma)$, where $\chi$ is the Euler characteristic. If $\Gamma$ is connected, we say the tropical curve is \emph{irreducible}, and its genus is equal to its first Betti number.

The tropical curves considered in the paper do not have vertex genus. Such a generalization would be needed if we were to consider descendant invariants, as in \cite{cavalieri2021counting}.

\subsubsection{Parametrized tropical curves.} In the following, we write $\TT M_\delta$ for one of the two tropical M\"obius strips $\TT M_0$ or $\TT M_1$, and $\TT C_\delta$ for the associated two to one cover.

\begin{defi}
A \emph{parametrized tropical curve} inside $\TT M_\delta$ (resp. $\TT C_\delta$, $\RR^2$) is a map $h:\Gamma\to\TT M_\delta$ where $\Gamma$ is an abstract tropical curve, $h$ is affine with integer slope on the edges of $\Gamma$ such that at each vertex, the tropical curve is balanced, i.e. the sum of the outgoing slopes of $h$ is $0$.
\end{defi}

Here, \textit{slope} is the derivative $h'_e$ of $h$ along the edge $e$, and \textit{integer} means that $h'_e$ is in $\ZZ^2$ for any affine chart of the M\"obius strip. The \textit{weight} of an edge is the integral length of its slope.

The lattice structure on the tropical M\"obius strip defined above induces a monodromy. This induced symmetry preserves the direction of some vectors in $\RR^2$, determined by the eigenspaces of the matrices in Proposition \ref{prop: definition moebius strip}. For both $\TT M_\delta$, the vertical direction $\left(\begin{smallmatrix} 0 \\ 1 \\ \end{smallmatrix}\right)$ is preserved. For $\TT M_0$, the other eigenspace is the span of $\left(\begin{smallmatrix} 1 \\ 0 \\ \end{smallmatrix}\right)$, and for $\TT M_1$ it is the span of $\left(\begin{smallmatrix} 2 \\ 1 \\ \end{smallmatrix}\right)$. Thus, considering the \emph{vertical direction} for curves is well-defined, and by the \emph{horizontal direction}, we mean the other eigenspace.

\begin{expl}
In Figure \ref{figure curves TM0}, we draw tropical curves in the M\"obius strip $\TT M_0$. We have two natural choices of fundamental domain. First, we can choose the fundamental domain to be $[0;l]\times\RR$, depicted in $(a)$, $(b)$ and $(c)$: the two sides of the strip are identified by a reflection along the horizontal axis and a translation, and slopes change accordingly .

\smallskip

Alternatively, we can choose $[0;2l]\times\RR_{\geqslant 0}$ as a fundamental domain. It is obtained from the first fundamental domain by cutting along the dots in the first row of Figure \ref{figure curves TM0}, and gluing it back along the right vertical side. We obtain a domain resembling the cylinders from \cite{blomme2021floor} but infinite  only in one direction. This is depicted in $(a')$, $(b')$ and $(c')$. We call the bottom finite side of the M\"obius strip its \emph{soul}.
\end{expl}

\begin{figure}[h]
\begin{center}
\begin{tabular}{ccc}
\begin{tikzpicture}[line cap=round,line join=round,>=triangle 45,x=0.3cm,y=0.3cm,scale = 0.85]
\clip(-1,-8) rectangle (11,8);
\draw [line width=1pt] (0,-8)-- (0,8) node[midway,sloped] {$>>$};
\draw [line width=1pt] (10,-8)-- (10,8) node[midway,sloped] {$<<$};
\draw [line width=1pt,dotted] (0,0) -- (10,0);

\draw [line width=2pt] (0,-3)--++(1,0)--++(6,6)--(10,3);
\draw [line width=2pt] (1,-8)--(1,-3);
\draw [line width=2pt] (7,3)--(7,8);
\end{tikzpicture}
& \begin{tikzpicture}[line cap=round,line join=round,>=triangle 45,x=0.3cm,y=0.3cm,scale = 0.85]
\clip(-1,-8) rectangle (11,8);
\draw [line width=1pt] (0,-8)-- (0,8) node[midway,sloped] {$>>$};
\draw [line width=1pt] (10,-8)-- (10,8) node[midway,sloped] {$<<$};
\draw [line width=1pt,dotted] (0,0) -- (10,0);

\draw [line width=2pt] (0,2)--++(6,-6)--++(2,0)--++(2,2);
\draw [line width=2pt] (6,-8)--(6,-4);
\draw [line width=2pt] (8,-8)--(8,-4);
\end{tikzpicture}
& \begin{tikzpicture}[line cap=round,line join=round,>=triangle 45,x=0.3cm,y=0.3cm,scale = 0.85]
\clip(-1,-8) rectangle (11,8);
\draw [line width=1pt] (0,-8)-- (0,8) node[midway,sloped] {$>>$};
\draw [line width=1pt] (10,-8)-- (10,8) node[midway,sloped] {$<<$};
\draw [line width=1pt,dotted] (0,0) -- (10,0);

\draw [line width=2pt] (0,4)--++(2,0)--++(1,1)--++(1,-1)--++(4,0)--++(2,-2);
\draw [line width=2pt] (0,-2)--++(2,2)--++(2,0)--++(4,-4)--++(2,0);
\draw [line width=2pt] (2,0)--(2,4);
\draw [line width=2pt] (4,0)--(4,4);
\draw [line width=2pt] (3,5)--(3,8) node[midway,right] {$2$};
\draw [line width=2pt] (8,4)--(8,8);
\draw [line width=2pt] (8,-8)--(8,-4);
\end{tikzpicture}\\
$(a)$ & $(b)$ & $(c)$ \\
\begin{tikzpicture}[line cap=round,line join=round,>=triangle 45,x=0.3cm,y=0.3cm,scale = 0.85]
\clip(-1,-1) rectangle (11,12);
\draw [line width=1pt] (0,0)-- (0,15) node[midway,sloped] {$>>$};
\draw [line width=1pt] (10,0)-- (10,15) node[midway,sloped] {$>>$};
\draw [line width=1pt] (0,0) node {$\bullet$} -- (5,0) node[midway,sloped] {$>$} node {$\bullet$} -- (10,0) node[midway,sloped] {$>$} node {$\bullet$};

\draw [line width=2pt] (1,0)--++(2,2)--++(1,0)--++(2,-2);
\draw [line width=2pt] (3,2)--(3,15);
\draw [line width=2pt] (4,2)--(4,15);
\end{tikzpicture}
& \begin{tikzpicture}[line cap=round,line join=round,>=triangle 45,x=0.3cm,y=0.3cm,scale = 0.85]
\clip(-1,-1) rectangle (11,12);
\draw [line width=1pt] (0,0)-- (0,15) node[midway,sloped] {$>>$};
\draw [line width=1pt] (10,0)-- (10,15) node[midway,sloped] {$>>$};
\draw [line width=1pt] (0,0) node {$\bullet$} -- (5,0) node[midway,sloped] {$>$} node {$\bullet$} -- (10,0) node[midway,sloped] {$>$} node {$\bullet$};

\draw [line width=2pt] (0,4)--++(2,-2)--++(5,0)--++(3,3);
\draw [line width=2pt] (0,5)--++(1,1)--++(7,0)--++(2,-2);
\draw [line width=2pt] (2,0)--(2,2);
\draw [line width=2pt] (7,0)--(7,2);
\draw [line width=2pt] (1,6)--(1,12);
\draw [line width=2pt] (8,6)--(8,12);
\end{tikzpicture}
& \begin{tikzpicture}[line cap=round,line join=round,>=triangle 45,x=0.3cm,y=0.3cm,scale = 0.85]
\clip(-1,-1) rectangle (11,12);
\draw [line width=1pt] (0,0)-- (0,15) node[midway,sloped] {$>>$};
\draw [line width=1pt] (10,0)-- (10,15) node[midway,sloped] {$>>$};
\draw [line width=1pt] (0,0) node {$\bullet$} -- (5,0) node[midway,sloped] {$>$} node {$\bullet$} -- (10,0) node[midway,sloped] {$>$} node {$\bullet$};

\draw [line width=2pt] (1,0)--++(1,1)--++(3,0)--++(1,-1);
\draw [line width=2pt] (0,4)--++(2,0)--++(1,1)--++(1,-1)--++(1,0)--++(1,1)--++(2,0)--++(1,-1)--++(1,0);
\draw [line width=2pt] (2,1)--(2,4);
\draw [line width=2pt] (4,0)--(4,4);
\draw [line width=2pt] (5,1)--(5,4);
\draw [line width=2pt] (9,0)--(9,4);
\draw [line width=2pt] (3,5)--(3,12) node[midway,right] {$2$};
\draw [line width=2pt] (6,5)--(6,12);
\draw [line width=2pt] (8,5)--(8,12);
\end{tikzpicture}\\
$(a')$ & $(b')$ & $(c')$ \\
\end{tabular}
\caption{\label{figure curves TM0}Examples of tropical curves inside the M\"obius strip $\TT M_0$. In the first row we take the fundamental domain to be $[0;l]\times\RR$, and in the second row $[0;2l]\times\RR_{\geqslant 0}$.}
\end{center}
\end{figure}
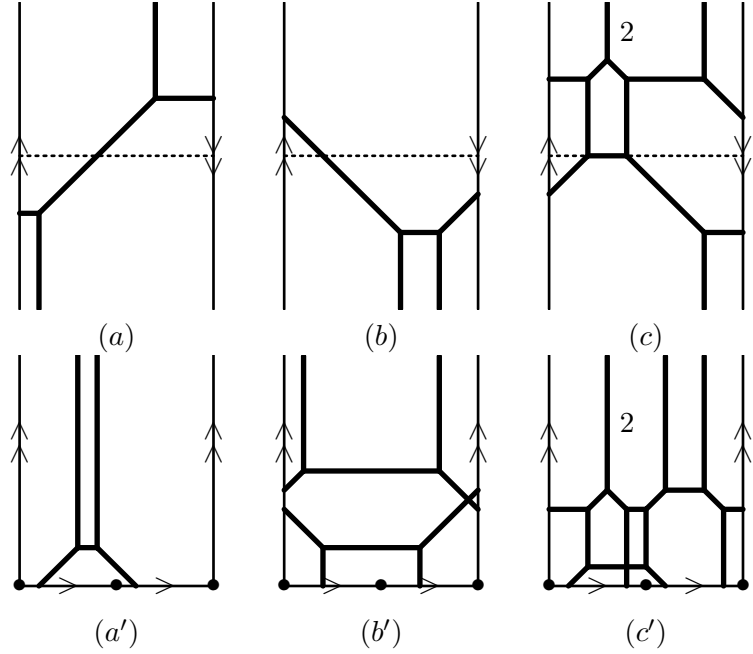

\begin{expl}
We proceed similarly for $\TT M_1$. Taking $[0;l]\times\RR$ as a fundamental domain, the gluing is the same as for $\TT M_0$, but the monodromy of the lattice structure is different. We give examples in the first row of Figure \ref{figure curves TM1}. Choosing  the fundamental domain $\{0\leqslant x\leqslant 2l,2y\geqslant x\}$, or $\{0\leqslant x\leqslant 2l,y\geqslant \max(0,x-l)\}$, we give examples in the second row of Figure \ref{figure curves TM1}.
\end{expl}

\begin{figure}[h]
\begin{center}
\begin{tabular}{ccc}
\begin{tikzpicture}[line cap=round,line join=round,>=triangle 45,x=0.3cm,y=0.3cm,scale = 0.85]
\clip(-1,-10) rectangle (11,10);
\draw [line width=1pt] (0,-10)-- (0,10) node[midway,sloped] {$>>$};
\draw [line width=1pt] (10,-10)-- (10,10) node[midway,sloped] {$<<$};
\draw [line width=1pt,dotted] (0,0) -- (10,0);

\draw [line width=2pt] (0,-3)--++(6,6)--++(4,0);
\draw [line width=2pt] (6,3)--(6,10);
\end{tikzpicture}
& \begin{tikzpicture}[line cap=round,line join=round,>=triangle 45,x=0.3cm,y=0.3cm,scale = 0.85]
\clip(-1,-10) rectangle (11,10);
\draw [line width=1pt] (0,-10)-- (0,10) node[midway,sloped] {$>>$};
\draw [line width=1pt] (10,-10)-- (10,10) node[midway,sloped] {$<<$};
\draw [line width=1pt,dotted] (0,0) -- (10,0);

\draw [line width=2pt] (0,2.5)--++(10,0);
\draw [line width=2pt] (0,-2.5)--++(10,10);
\draw [line width=2pt] (0,-7.5)--(5,-7.5)--++(5,5);
\draw [line width=2pt] (5,-7.5)--(5,-10);
\end{tikzpicture}
& \begin{tikzpicture}[line cap=round,line join=round,>=triangle 45,x=0.3cm,y=0.3cm,scale = 0.85]
\clip(-1,-10) rectangle (11,10);
\draw [line width=1pt] (0,-10)-- (0,10) node[midway,sloped] {$>>$};
\draw [line width=1pt] (10,-10)-- (10,10) node[midway,sloped] {$<<$};
\draw [line width=1pt,dotted] (0,0) -- (10,0);

\draw [line width=2pt] (0,1)--++(2,2)--++(4,-4)--++(4,0);
\draw [line width=2pt] (0,6)--++(2,-2)--++(1,1)--++(4,0)--++(3,3);
\draw [line width=2pt] (0,-8)--++(8,0)--++(2,4);
\draw [line width=2pt] (8,-10)-- node[midway,right] {$2$} (8,-8);
\draw [line width=2pt] (6,-10)--(6,-1);
\draw [line width=2pt] (2,3)--(2,4) node[midway,right] {$2$};
\draw [line width=2pt] (3,5)--(3,10);
\draw [line width=2pt] (7,5)--(7,-10);
\end{tikzpicture}\\
$(a)$ & $(b)$ & $(c)$ \\
\begin{tikzpicture}[line cap=round,line join=round,>=triangle 45,x=0.3cm,y=0.3cm,scale = 0.85]
\clip(-1,-1) rectangle (11,16);
\draw [line width=1pt] (0,0)-- (0,15) node[midway,sloped] {$>>$};
\draw [line width=1pt] (10,5)-- (10,20) node[midway,sloped] {$>>$};
\draw [line width=1pt] (0,0) node {$\bullet$} -- (5,2.5) node[midway,sloped] {$>$} node {$\bullet$} -- (10,5) node[midway,sloped] {$>$} node {$\bullet$};

\draw [line width=2pt] (2,1)--++(2.5,2.5)--(7,3.5);
\draw [line width=2pt] (4.5,3.5)--(4.5,15);
\end{tikzpicture}
& \begin{tikzpicture}[line cap=round,line join=round,>=triangle 45,x=0.3cm,y=0.3cm,scale = 0.85]
\clip(-1,-1) rectangle (11,16);
\draw [line width=1pt] (0,0)-- (0,15) node[midway,sloped] {$>>$};
\draw [line width=1pt] (10,5)-- (10,20) node[midway,sloped] {$>>$};
\draw [line width=1pt] (0,0) node {$\bullet$} -- (5,0) node[midway,sloped] {$>$} node {$\bullet$} -- (10,5) node[midway,sloped] {$>$} node {$\bullet$};

\draw [line width=2pt] (2,0)--++(2,2)--++(3,0);
\draw [line width=2pt] (4,2)--(4,15);
\end{tikzpicture}
& \begin{tikzpicture}[line cap=round,line join=round,>=triangle 45,x=0.3cm,y=0.3cm,scale = 0.85]
\clip(-1,-1) rectangle (11,16);
\draw [line width=1pt] (0,0)-- (0,15) node[midway,sloped] {$>>$};
\draw [line width=1pt] (10,5)-- (10,20) node[midway,sloped] {$>>$};
\draw [line width=1pt] (0,0) node {$\bullet$} -- (5,0) node[midway,sloped] {$>$} node {$\bullet$} -- (10,5) node[midway,sloped] {$>$} node {$\bullet$};

\draw [line width=2pt] (1,0)--++(1,1)--++(4,0);
\draw [line width=2pt] (2,0)--++(2,2)--++(3,0);
\draw [line width=2pt] (0,4)--++(2,0)--++(1,1)--++(1,2)--++(1,3)--++(2,0)--++(1,-1)--++(2,0);
\draw [line width=2pt] (2,1)--(2,4);
\draw [line width=2pt] (3,0)--(3,5);
\draw [line width=2pt] (4,2)--(4,7);
\draw [line width=2pt] (8,3)--(8,9);
\draw [line width=2pt] (5,10)--(5,15) node[midway,left] {$3$};
\draw [line width=2pt] (7,10)--(7,15);
\end{tikzpicture}\\
$(a')$ & $(b')$ & $(c')$ \\
\end{tabular}
\caption{\label{figure curves TM1}Examples of tropical curves inside the M\"obius strip $\TT M_1$.}
\end{center}
\end{figure}
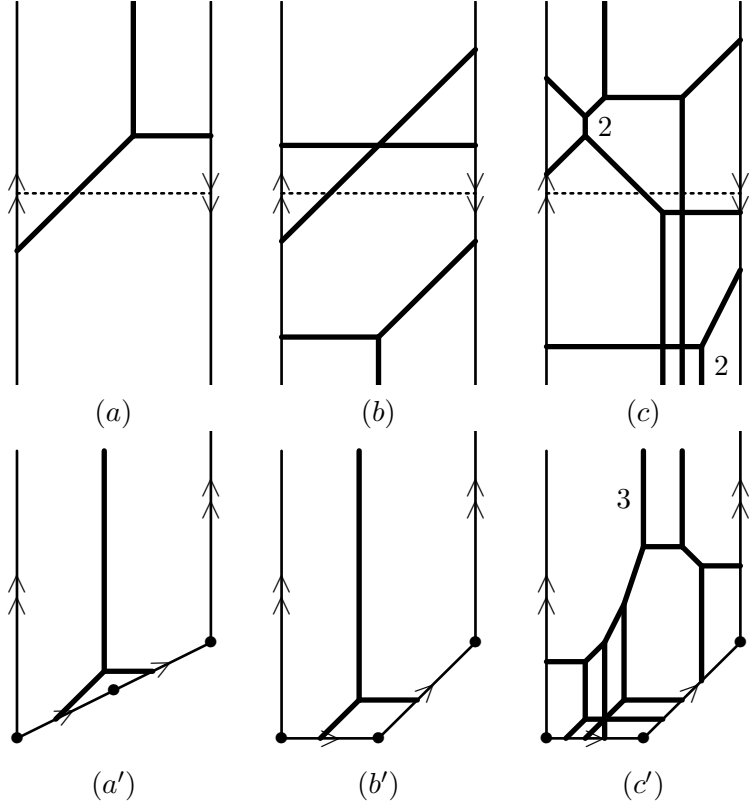

\subsubsection{Degree of a tropical curve.} The degree of a parametrized tropical curve is defined as the class inside the tropical homology group $H_{1,1}(\TT M_\delta,\ZZ)$ realized by the tropical cycle. See \cite{itenberg2019tropical} for a definition of the tropical homology groups. In our case, this homology group is isomorphic to $\ZZ^2$. Moreover, this group has a well-defined intersection form which is unimodular by tropical Poincar\'e duality. It is thus possible to recover the degree of a tropical curve by intersecting it with cycles that form a basis of $H_{1,1}(\TT M_\delta,\ZZ)$.

Both homology groups contain the class $F$ of a fiber (red in Figure \ref{figure computation degree}), and the class $E$ (blue in Figure \ref{figure computation degree}) of the boundary of the M\"obius strip. These two classes satisfy $E\cdot F=2$ since a fiber intersects the boundary of the M\"obius strip twice. Thus, $E$ and $F$ span an index $2$ sublattice of $H_{1,1}(\TT M_\delta,\ZZ)$. They also form a basis of the group if we allow the coefficients to be rational. To describe  $H_{1,1}(\TT M_\delta,\ZZ)$ fully, we need to consider the cases of $\TT M_0$ and $\TT M_1$ separately.

\begin{itemize}[label=$\ast$,noitemsep]
\item For $\TT M_0$, the M\"obius strip contains the tropical elliptic curve that goes around the strip once (in horizontal direction), contained in the soul of the M\"obius strip. As this curve intersects the fiber $F$ only once and  does not intersect the boundary, its class can be expressed as $\frac{1}{2}E$. Thus, elements of $H_{1,1}(\TT M_0,\ZZ)$ can be written as
$$aE+bF,\ a\in\frac{1}{2}\ZZ,b\in\ZZ.$$

\item For $\TT M_1$, we have the class $C_0$ of a curve as depicted in Figure \ref{figure curves TM1} $(a)$. It has intersection $1$ with both $E$ and $F$, thus $C_0=\frac{E+F}{2}$. Hence, elements in $H_{1,1}(\TT M_1,\ZZ)$ can be written as
$$aE+bF,\ a,b\in\frac{1}{2}\ZZ,a+b\in\ZZ.$$
\end{itemize}

The description can be unified in the following way: the lattice is the set of $aE+bF$ with $a,b\in\frac{1}{2}\ZZ$ such that
$$2b\equiv 2\delta a\ \mathrm{mod}\ 2.$$
In both cases, a curve in the class $aE+bF$ has $2b$ ends (counted with weights) since its intersection with $E$ is equal to $2b$, and its intersection with $F$ is equal to $2a$.

\begin{expl}\label{ex degrees on TM0}
We consider the curves in Figure \ref{figure curves TM0} on $\TT M_0$. We compute their degree by intersecting with $E$ and $F$, obtaining respectively
	$$\begin{array}{rlrl}
	(a) & \frac{1}{2}E+F & (a') & \frac{1}{2}E+F \\
	(b) & \frac{1}{2}E+F & (b') & 2E+F \\
	(c) & E+2F & (c') & \frac{3}{2}E+2F . \\
	\end{array}$$
Notice that when considering the second fundamental domain, fibers have two ends going to top infinity, both meeting the top cycle, as depicted in Figure \ref{figure computation degree}.
\end{expl}

\begin{figure}
\begin{center}
\begin{tikzpicture}[line cap=round,line join=round,>=triangle 45,x=0.3cm,y=0.3cm,scale = 0.85]
\clip(-1,-1) rectangle (11,12);
\draw [line width=1pt] (0,0)-- (0,15) node[midway,sloped] {$>>$};
\draw [line width=1pt] (10,0)-- (10,15) node[midway,sloped] {$>>$};
\draw [line width=1pt] (0,0) node {$\bullet$} -- (5,0) node[midway,sloped] {$>$} node {$\bullet$} -- (10,0) node[midway,sloped] {$>$} node {$\bullet$};
\draw[line width=2pt, blue](0,11.75)--(10,11.75);

\draw [line width=2pt] (1,0)--++(1,1)--++(3,0)--++(1,-1);
\draw [line width=2pt] (0,4)--++(2,0)--++(1,1)--++(1,-1)--++(1,0)--++(1,1)--++(2,0)--++(1,-1)--++(1,0);
\draw [line width=2pt] (2,1)--(2,4);
\draw [line width=2pt] (4,0)--(4,4);
\draw [line width=2pt] (5,1)--(5,4);
\draw [line width=2pt] (9,0)--(9,4);
\draw [line width=2pt] (3,5)--(3,12) node[midway,right] {$2$};
\draw [line width=2pt] (6,5)--(6,12);
\draw [line width=2pt] (8,5)--(8,12);

\draw [line width=2pt,red] (2.5,0)--(2.5,12);
\draw [line width=2pt,red] (7.5,0)--(7.5,12);
\end{tikzpicture}
\end{center}
\caption{\label{figure computation degree} A curve in the class $\frac{3}{2}E+2F$ in $\TT M_0$, along with a (red) fiber of the projection to $E=\RR/l\ZZ$ and the boundary $E_{\infty}$ marked in blue.}
\end{figure}
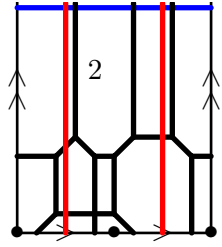

\begin{expl}\label{ex degrees on TM1}
The classes of the curves in $\TT M_1$ depicted in Figure \ref{figure curves TM1} are as follows:
$$\begin{array}{rlrl}
	(a) & C_0=\frac{1}{2}E+\frac{1}{2}F & (a') & C_0 \\
	(b) & E+C_0=\frac{3}{2}E+\frac{1}{2}F & (b') & C_0 \\
	(c) & E+2F+C_0 = \frac{3}{2}E+\frac{5}{2}F & (c') & 2E+2F . \\
	\end{array}$$
\end{expl}

\begin{defi}
A parametrized tropical curve in the class $aE+bF$ is said to have tangency profile $\mu\vdash 2b$ (i.e. $\mu$ is a partition of $2b$) if it has $\mu_i$ ends of weight $i$ for each $i$.
\end{defi}

\begin{expl}
The curves in Figure \ref{figure curves TM0} $(c)$ and $(c')$ have tangency profile $\mu=2^1 1^2$, and the curve in Figure \ref{figure curves TM1} $(c')$ has tangency profile $\mu=3^1 1^1$.
\end{expl}

\subsubsection{Cutting procedure.}\label{sec Cutting procedure} As in case of cylinders \cite{blomme2021floor} and abelian surfaces \cite{blomme2022abelian1}, we can construct a parametrized tropical curve $\widehat{h}:\widehat{\Gamma}\to\RR^2$ on $\RR^2$ from a parametrized tropical curve  $h:\Gamma\to\TT M_\delta$. To do so, we remove a finite set of points $\Q$ from $\Gamma$ such that $h$ restricted to the complement of $\P$ can be lifted to $\RR^2$,  the universal cover of the M\"obius strip. Such a set is called \textit{addmissible}. We then extend edges cut in the process to get ends. We refer to the corresponding sections of \cite{blomme2021floor} and \cite{blomme2022abelian1} for more details.

\begin{rem}\label{remark induced menelaus}
    In the toric setting as well as in the cylinder case, we have a \textit{Menelaus relation} between the positions of the ends of a tropical curve. In the toric  case \cite[Proposition 39]{mikhalkin2017quantum}, the relation comes from the balancing condition. More precisely, let $\Gamma$ be a tropical curve in $\RR^2$. Then, each end $e$ has an associated \textit{moment}, defined as the determinant of the outgoing slope of $e$ and any point $p\in e$. The Menelaus theorem asserts that the sum of moments is $0$. For cylinders \cite{blomme2021floor}, it follows by applying the toric tropical Menelaus relation to the tropical curve in $\RR^2$ obtained by cutting as above. Since these relations come from the existence of an area form, there is a priori no Menelaus relation in the M\"obius strip case. 
   
    Using the cutting procedure, we can still use the information of the Menelaus relation between the positions of ends in vertical direction in $\RR^2$ to infer information on the positions of ends in vertical direction on the M\"obius strip. We call the relations arising from this procedure the \emph{induced Menelaus relations} on the M\"obius strip. For curves in tropical M\"obius strips, we define the moment of an end to be the coordinate of its projection onto the elliptic curve. 
\end{rem}

\subsection{Dimension of the moduli space of curves}

We compute the dimension of the deformation space of tropical curves in a M\"obius strip $\TT M_\delta$. We will only consider parametrized curves $h:\Gamma\to\TT M_\delta$ for which $h$ is an immersion.

\begin{defi}
A parametrized tropical curve $h:\Gamma\to\TT M_\delta$ is \textit{simple} if $\Gamma$ is trivalent and $h$ is an immersion.
\end{defi}
 In particular, a simple tropical curve has no contracted edge (edge with slope $0$) nor flat vertex, which is a vertex all of whose adjacent edges have the same slope.
 
As previously done for cylinders, \cite{blomme2021floor}, we can characterize curves without ends.

\begin{prop}
Let $h:\Gamma\to\TT M_\delta$ be a parametrized tropical curve for which $h$ is an immersion, and that has no end. Then $\Gamma$ has genus $1$, a unique (bounded) edge of slope $\left(\begin{smallmatrix} w \\ 0 \\ \end{smallmatrix}\right)$ if on $\TT M_0$, or slope $\left(\begin{smallmatrix} 2w \\ w \\ \end{smallmatrix}\right)$ if on $\TT M_1$.
\end{prop}

\begin{proof}
Let $\TT C_\delta\to\TT M_\delta$ be the two-to-one cover of the M\"obius strip by a tropical cylinder $\TT C_\delta$. Up to taking a two-to-one cover of $\Gamma$, we can lift the tropical curve $h:\Gamma\to\TT M_\delta$ to a tropical curve $\widetilde{h}:\widetilde{\Gamma}\to\TT C_\delta$. The lift $\widetilde{h}$ is an immersion and $\widetilde{\Gamma}$ still has no end. These have been characterized in \cite{blomme2021floor}:
\begin{itemize}[label=-]
\item For $\TT C_0$, they are of the form $t\mapsto (w\cdot t,c)\in\TT C_0$, where $c\in\RR$ and $n\in\NN$. It is a curve that goes around the cylinder direction $k$ times with slope $\left(\begin{smallmatrix} w \\ 0 \\ \end{smallmatrix}\right)$.
\item For $\TT C_1$, they are of the form $t\mapsto (2w\cdot t,w\cdot t)$.
\end{itemize}
Both maps are periodic, and can be made compact by quotienting by a sublattice of the periods.
\end{proof}
In particular, we get two types of genus $1$ curves without ends: fixed curves meeting the soul of the M\"obius strip an odd number of times, and curves that have a 1-dimensional deformation space. The latter are obtained as image of curves without ends in the cover by $\TT C_\delta$.

 For what follows, let $\M_{g,n}(\TT M_\delta,(a,b),\mu)$ be the moduli space of genus $g$ parametrized tropical curves in the class $aE+bF$ with tangency profile $\mu\vdash 2b$ and $n$ marked points in $\TT M_\delta$. It is a polyhedral complex whose faces correspond to combinatorial types of maps, i.e. shapes of graphs $\Gamma$ with a choice of slope of $h$ on them. Each curve can be parametrized by the choice of edge lengths and the image of a vertex inside $\TT M_\delta$.

\begin{prop}\label{prop description moduli space}
The dimension of the subspace of $\M_{g,n}(\TT M_\delta,(a,b),\mu)$ parametrizing  simple tropical curves is $|\mu|+g-1+n$, where $|\mu|=\sum\mu_i$ is the length of the partition $\mu\neq\emptyset$.
\end{prop}

\begin{proof}
The proof is similar to the computation of the dimension in \cite{blomme2021floor}. Let $h:\Gamma\to\TT M_\delta$ be a simple tropical curve. Using the cutting procedure described in Section \ref{sec Cutting procedure} with an admissible set $\Q$, we get a parametrized tropical curve $\widehat{h}:\widehat{\Gamma}\to\RR^2$ with new marked 
points coming from $\Q$ on the non-vertical ends that get identified through the quotient map. A small deformation of $\Gamma$ is equivalent to a small deformation of $\widehat{\Gamma}$ where marked ends keep their identification under the quotient.

Let $e$ and $e'$ being two ends identified through the quotient map, $\mu_e$ and $\mu_{e'}$ be their moments. Let $\lambda_e\in\ZZ$ be the class in $\pi_1(\TT M_\delta)$ realized by the following loop: push down to $\TT M_\delta$ a path in $\widehat{\Gamma}$ between the two new marked points on the ends. The condition for ends to keep be identified throughout the deformation of $\widehat{\Gamma}$ can be written as $\mu_e+(-1)^{\lambda_e}\mu_{e'}=0.$

We have such a condition for each pair of ends. As there is at least one vertical end, the above impose $|\Q|$ linearly independent relations. By the computation of the dimension in the planar case from \cite{mikhalkin2005enumerative}, $\widehat{\Gamma}$ varies in a space of dimension $(|\mu|+2|\Q|)+(g-|\Q|)-1+n$. Substracting $|\Q|$ yields the expected dimension.
\end{proof}

\section{Tropical invariance and refined enumeration}

\subsection{Enumerative problem and geometry of the solutions}

\subsubsection{Enumerative problems.} The moduli space of curves on $\TT M_\delta$ has the evaluation map 
$$\mathrm{ev}:\M_{g,n}(\TT M_\delta,(a,b),\mu)\longrightarrow \TT M_\delta^n\times \TT E^{|\mu|},$$
evaluating the position of marked points and ends. For a suitable choice of $|\mu|$ and $n$, the dimension of the codomain of $\mathrm{ev}$ agrees with the dimension of the subspace of $\M_{g,n}(\TT M_\delta,(a,b),\mu)$ that parametrizes simple combinatorial types. Hence, we can pose the following enumerative problems:
\begin{itemize}[label=-]
\item How many genus $g$ curves in the class $aE+bF$ pass through $2b+g-1$ points in generic position?
\end{itemize}
Choosing partitions $\mu$ and $\nu$ such that $\|\mu+\nu\|=\sum i(\mu_i+\nu_i)=2b$, we can ask a relative version:
\begin{itemize}[label=-]
\item How many genus $g$ curves in the class $aE+bF$ with tangency profile $\mu+\nu$ pass through $|\nu|+g-1$ points and have $\mu_i$ ends of weight $i$ of fixed position?
\end{itemize}

In each case, the problem amounts to finding the preimages of a general point in the codomain of the evaluation map. 

\begin{prop}
For a generic choice of constraints satisfying the conditions above, there is a finite number of preimages of $\mathrm{ev}$ in  $\M_{g,n}(\TT M,(a,b),\mu)$, all of which are simple.
\end{prop}

\begin{proof}
Let $h:\Gamma\to\TT M_\delta$ be a tropical curve. Up to a reparametrization by merging parallel edges, we can assume that $h$ is an immersion. This operation replaces $\Gamma$ by a graph that is either of smaller genus, has fewer ends, or at least one non-trivalent vertex. This does not change the dimension of the image under the evaluation map. Thus, the evaluation map is not surjective for these combinatorial types. hence, a generic choice of constraints is out of their image, and only simple combinatorial types provide solutions.

For simple combinatorial types, as the codomain has the same dimension, the set of preimages is discrete. Indeed, if the derivative of $\text{ev}$ is not injective, it is not surjective either, and a general choice of constraints would not be in the image.

The other combinatorial types where $h$ is an immersion vary in a space of strictly smaller dimension, thus cannot contribute any solution, hence the result.
\end{proof}

\subsubsection{Position of the marked points on the solutions.} In the case of $\RR^2$ \cite{mikhalkin2005enumerative}, as well as the case of cylinders \cite{blomme2021floor}, we can characterize the position of marked points on a tropical curve solving the problem. Given a parametrized curve meeting the constraints above, the complement of marked points on it is without cycle  and each connected component contains a unique end.
This is due to the fact that both cycles and paths relating two distinct ends can deform in a dimension one space, which is prohibited by finiteness of the number of solutions. Thus, the complement of marked points is a forest of trees, each rooted at an end, whose leaves are marked points.

\smallskip

In the case of M\"obius strips, the description differs due to the non-orientability of the surface, as disorienting cycles cannot be deformed without moving at least one of the adjacent edges.

\begin{prop}\label{proposition shape of solution}
Let $h:\Gamma\to\TT M_\delta$ be a simple tropical curve of genus $g$ in the class $aE+bF$ passing through a generic configuration of $2b+g-1$ points $\P$. Then, each component of the complement of marked points $h(\Gamma)\setminus \P$ satisfies either of the following conditions:
\begin{enumerate}[label=$(\roman*)$]
\item It contains a unique end and no cycle, or
\item It contains a unique disorienting cycle and no end.
\end{enumerate}
\end{prop}

\begin{proof}
We obtain a space $\Gamma_\P$ by disconnecting the curve at each marked point, replacing the edge it lies on by two closed half-edges. The connected components of $\Gamma_\P$ correspond exactly to the connected components of $\Gamma\setminus h^{-1}(\P)$. As there are  $2b+g-1$ points in $\P$, and the Euler characteristic of $\Gamma$ is $1-g-2b$, we have
$$\chi(\Gamma_\P)=  \chi(\Gamma)+ 2b+g-1 = 1-g-2b +2b+g-1 = 0.$$
Let $\Gamma_i$ be a connected component of $\Gamma_\P$. The Euler characteristic of $\Gamma_i$ is $1-g_i-x_i$, where $g_i$ is the first Betti number of the component, and $x_i$ the number of ends that it contains.

If we have $g_i=x_i=0$, the component is a tree  with a point of $\P$ at each of its outgoing edges since it contains no end. Here, the two leaves incident to the same marked point can belong to the same component. As the genus is $0$, we can lift the component to $\RR^2$. Then, we have the tropical Menelaus condition from Remark \ref{remark induced menelaus} between the position of ends, yielding a relation between the position of the marked points. Notice that if the two half-edges coming from the same edge appear as leafs of the component, their contributions to the relation cancel. First, assume the relation is nonempty. Then, this corresponds to a non-generic choice of the constraints, contradicting genericity.

Now, assume the relation is empty. Then, all boundary points come in pairs, i.e., for a leaf corresponding to a marked point, the other leaf associated to the same marked point is also adjacent to the component. Therefore, the component corresponds to a curve without ends. It is thus a genus $1$ curve that varies in a $1$-dimensional space containing a unique point condition. This closes the proof if $\Gamma$ is connected.

Thus, we can assume that we always have $1-g_i-x_i\leqslant 0$. As the sum is $0$, each summand is $0$ and we have two possibilities: either $g_i=0$ and $x_i=1$; or $g_i=1$ and $x_i=0$.

The first case corresponds to components without cycles and with a unique end. For the second case, as it is always possible to deform an orienting cycle, the unique cycle has to be disorienting. This concludes the proof.
\end{proof}

\begin{rem}
Proposition \ref{proposition shape of solution} essentially states that disorienting cycles behave like ends as they can be deformed to comply with the deformation of an adjacent edge.
\end{rem}
\begin{figure}
\begin{center}
\begin{tabular}{cc}\begin{tikzpicture}[line cap=round,line join=round,>=triangle 45,x=0.4cm,y=0.4cm,scale = 0.85]
\clip(-1,-1) rectangle (11,12);
\draw [line width=1pt] (0,0)-- (0,15) node[midway,sloped] {$>>$};
\draw [line width=1pt] (10,0)-- (10,15) node[midway,sloped] {$>>$};
\draw [line width=1pt] (0,0) node {$\bullet$} -- (5,0) node[midway,sloped] {$>$} node {$\bullet$} -- (10,0) node[midway,sloped] {$>$} node {$\bullet$};

\draw [line width=2pt,red] (1,0)--++(2,2)--++(1,0)--++(2,-2);
\draw [line width=2pt,red] (3,2)--(3,15);
\draw [line width=2pt,red] (4,2)--(4,15);

\draw [line width=2pt,blue] (1.7,0)--++(1.5,1.5)--++(1.8,0)--++(1.5,-1.5);
\draw [line width=2pt,blue] (3.2,1.5)--(3.2,15);
\draw [line width=2pt,blue] (5,1.5)--(5,15);
\end{tikzpicture} &
\begin{tikzpicture}[line cap=round,line join=round,>=triangle 45,x=0.4cm,y=0.4cm,scale = 0.85]
\clip(-1,-1) rectangle (11,12);
\draw [line width=1pt] (0,0)-- (0,15) node[midway,sloped] {$>>$};
\draw [line width=1pt] (10,5)-- (10,20) node[rotate = 90] at (10,7.5) {$>>$};
\draw [line width=1pt] (0,0) node {$\bullet$} -- (5,0) node[midway,sloped] {$>$} node {$\bullet$} -- (10,5) node[midway,sloped] {$>$} node {$\bullet$};

\draw [line width=2pt, blue] (2,0)--++(2,2)--++(3,0);
\draw [line width=2pt, blue] (4,2)--(4,15);

\draw [line width=2pt,red] (3,0)--++(3,3)--++(2,0);
\draw [line width=2pt,red] (6,3)--(6,15);
\end{tikzpicture} \\
$(a)$ & $(b)$ \\
\end{tabular}
\caption{\label{fig-deformation-disorienting-cycle} A disorienting cycle with a deformation moving only one adjacent edge. $(a)$ is on $\TT M_0$ and $(b)$ is on $\TT M_1$.}
\end{center}
\end{figure}
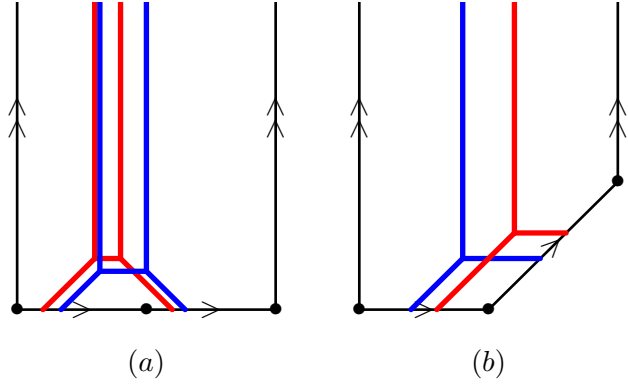
\begin{expl}
In Figure \ref{fig-deformation-disorienting-cycle}, we can see two tropical curves with a unique disorienting cycle. Contrarily to orienting cycles, we see that it is not possible to deform the cycle while fixing the adjacent edges, but it is possible to deform the cycle while varying only one of the adjacent edges: the right vertical edge for $(a)$ (over $\TT M_0$) and the only vertical edge for $(b)$ (over $\TT M_1$).
\end{expl}

\subsection{Multiplicities}
In  tropical enumerative geometry, we usually count parametrized tropical curves with multiplicity. The case of curves on tropical M\"obius strips is no different. Let $h:\Gamma\to\TT M_\delta$ be a parametrized tropical curve of genus $g$, tangency profile $\mu+\nu$ (where $\mu$ and $\nu$ are partitions) with $n=|\nu|+g-1$ generic fixed marked points $\P$, where the ends indexed by $\mu$ are fixed. By genericity, $h$ is a simple tropical curve. For an edge or end $e$, let $u_e$ be its slope, and $u_e=w_e\widehat{u_e}$, where $\widehat{u_e}$ is the primitive lattice vector in direction $u_e$. Using Proposition \ref{proposition shape of solution}, we  choose an orientation of the edges of $\Gamma$ such the edges and ends inside a component of $h(\Gamma)\setminus \P$ point toward the disorienting cycle or the free end in the component. For the edges of disorienting cycles, we choose an arbitrary orientation of the cycle.

Each M\"obius strip is endowed with a lattice structure as given in Proposition \ref{prop: definition moebius strip}. Thus, given a vertex $V$ of $\Gamma$, we obtain a rank $2$ lattice $N_V$ at the point $h(V)\in \TT M$. Moreover, given a bounded edge $e$ in $\Gamma$, the lattice structure can be trivialized on the interior of the edge, so that we can define a rank $2$ lattice $N_e$ containing the slope $u_e$. Moreover, for each flag $V\in e$ of $\Gamma$, we have a map $N_V\to N_e$ identifying both lattices. We now have the following lattice map
$$\begin{array}{r*2{>{\displaystyle}c}*2{>{\displaystyle}l}}
\Theta: & \bigoplus_{V\in V(\Gamma)} N_V  & \longrightarrow & \bigoplus_{e\in E_b(\Gamma)}N_e/\langle \widehat{u_e}\rangle\oplus\bigoplus_1^n N_{V_i} \oplus \bigoplus_\mu N_e/\langle \widehat{u_e}\rangle \\
 & (\phi_V) & \longmapsto & \left( (\phi_{\hfk(e)}-\phi_{\tfk(e)}),(\phi_{V_i}),(\phi_{\hfk(e)})\right)
\end{array},$$
where $\hfk(e)$ (resp. $\tfk(e)$ if $e$ is bounded) is the head (resp. tail) of an edge $e$, and $V_i$ is the vertex corresponding to the $i$-th marked point. The domain is indexed by the vertices of the tropical curve (including the marked points), and the codomain is indexed by the bounded edges of the curve along with the marked points and ends whose position is evaluated. The curve is simple, hence trivalent. Counting the number of pairs $(V\in e)$ in two ways and computing the Euler characteristic, we have two equations:
$$3|V(\Gamma)|=2|E_b(\Gamma)|+n+|\mu|+|\nu|\text{ and }|V(\Gamma)|-|E_b(\Gamma)|=1-g.$$
Thus, $\Gamma$ has $| V(\Gamma)|=|\mu|+|\nu|+2g-2+n$ vertices, and $| E_b(\Gamma)|=3g-3+|\mu|+|\nu|+n$ bounded edges. As $n=|\nu|+g-1$, both ranks are equal to $2|\mu|+4|\nu|+6g-6$. Thus, we can speak about the lattice index of the image of $\Theta$. Taking bases of the lattices, we can see $\Theta$ as an integer matrix whose lattice index is equal to $|\det\Theta|$.
\begin{defi}
We define the multiplicity of a tropical curve to be
$$m_\Gamma^\CC=|\det\Theta|\prod_{e\in E_b(\Gamma)}w_e.$$
\end{defi}
The following proposition gives a concrete expression for the multiplicity.
\begin{prop}\label{prop-computation-mult}
Let $h:\Gamma\to\TT M_\delta$ be a simple tropical curve of genus $g$ in the class $aE+bF$ of tangency profile $\mu+\nu\vdash 2b$ passing through a generic configuration of $|\nu|+g-1$ points $\P$ with the position of the ends indexed by $\mu$ fixed. Let $k$ be the number of disorienting cycles in the complement of the marked points. Then, the multiplicity of $\Gamma$ is given by:
$$m_\Gamma^\CC=\frac{2^k}{I^\mu}\prod_V m_V,$$
where $m_V$ is the usual vertex multiplicity, defined as $m_V=|\det(a_V,b_V)|$ if $a_V$ and $b_V$ are two out of the three outgoing slopes, and $I^\mu=\prod_i i^{\mu_i}$.
\end{prop}

\begin{proof}
The proof consists in computing the determinant of the lattice map $\Theta$. To write the matrix of $\Theta$, we choose a basis of each $N_V$. For each $e$, choosing an identification with $\ZZ^2$, the linear form $\det(\widehat{u_e},-)$ provides a coordinate (i.e. a bijection to $\ZZ$) of the quotient lattice $N_e/\langle \widehat{u_e}\rangle$. The matrix of $\Theta$ is constructed as follows.
\begin{itemize}[label=-]
\item For fixed ends, we have a $1\times 2$ block comprised of $\det(\widehat{u_e},-)$ evaluating the position of the unique adjacent vertex in the quotient.
\item For bounded edges, we have two of these blocks, one for each adjacent vertex. Substituting the primitive slope $\widehat{u_e}$ by the actual slope $u_e=w_e\widehat{u_e}$ divides the determinant by $w_e$. This cancels with the $w_e$ in the definition of $m_\Gamma^\CC$. Thus, we can assume that the slopes are the true slopes $u_e$, getting a map $\Theta'$ with $m_\Gamma^\CC=\frac{1}{I^\mu}|\det\Theta'|$ where we divide by the product of weights of marked ends since they do not appear in the original product.
\item Evaluating at each of the $|\nu|+g-1$ marked points, contributes $\left(\begin{smallmatrix} 1 & 0 \\ 0 & 1 \\ \end{smallmatrix}\right)$. Applying Laplace expansion for the determinant with respect to these deletes one of the $\det(u_e,-)$ in each row corresponding to an adjacent bounded edge. We are left with the matrix of the lattice map $\Theta$ for the tropical curve where each point has been removed, and the bounded edge containing it is replaced by a pair of ends whose positions we evaluate.
\end{itemize}
Thus, we may assume that there are no marked points. The rest of the computation is done recursively in two steps. The first step is needed to prove that the multiplicity in the sense of Nishinou-Siebert as in \cite{nishinou2006toric} coincides with the one defined by Mikhalkin in \cite{mikhalkin2005enumerative}. We briefly recall it for sake of completeness. The second step is specific to our case, dealing with disorienting cycles. Notice that the matrix of $\Theta$ splits into blocks for the different connected components of $\Gamma$, of which there might be several after the last step above.

We prune the branches of the complement of marked points. If $V$ is a vertex adjacent to two fixed ends with slopes $u_1$ and $u_2$, the column of $V$ only has the following blocks:
$$\begin{array}{|c|}
\hline
\det(u_1,-) \\
\hline
\det(u_2,-) \\
\hline
\det(u_e,-) \\
\hline
\end{array}.$$
The first two rows correspond to the evaluation of the position of the adjacent ends, and the last row to the position of the remaining adjacent edge $e$. This row does not appear if $e$ is an end. As the first two rows are the only non-zero elements, we expand $m_V=\det(u_1,u_2)\cdot|\det\Theta_{\Gamma\setminus V}|$ where $\Theta_{\Gamma\setminus V}$ is the lattice of the curve $h|_{\Gamma\setminus V}$, where $e$ disappears if it was unbounded or becomes an end whose position is evaluated. 

 If there are no cycles in the complement of the marked points, the previous step is enough to fully compute the determinant. Else, by Proposition \ref{proposition shape of solution} $\Gamma$ has a unique disorienting cycle, all of whose adjacent edges are ends whose position is evaluated. Let $V_1,\dots,V_p$ be the vertices on the cycle, $u_1,\dots,u_p$ be the slopes of the edges of the cycle, and $v_1,\dots,v_p$ the slope of the adjacent ends, such that $V_i$ is adjacent to the edges of slopes $u_i,v_i$ and $u_{i+1}$, indices taken modulo $p$. We have $u_i+v_i=u_{i+1}$. The matrix has the following form:
\begin{center}
\begin{tabular}{|ccccc|}
\hline
\multicolumn{1}{|c|}{$\det(u_1,-)$} &                         &                                  & \multicolumn{1}{c|}{}     & $\det(u_1,-)$ \\ \cline{2-2} \cline{5-5} 
\multicolumn{1}{|c|}{$-\det(u_2,-)$} & \multicolumn{1}{c|}{$\det(u_2,-)$} &       &                            &                               \\ \cline{1-1}
\multicolumn{1}{|c|}{}   & \multicolumn{1}{c|}{$-\det(u_3,-)$} & $\ddots$                      &                           &    \\ \cline{2-2} \cline{4-4}
                         &                          & \multicolumn{1}{c|}{$\ddots$} & \multicolumn{1}{c|}{$\det(u_{p-1},-)$} &    \\ \cline{5-5} 
                         &                               & \multicolumn{1}{c|}{}      & \multicolumn{1}{c|}{$-\det(u_p,-)$}   & $\det(u_p,-)$ \\ \hline
\multicolumn{1}{|c|}{$\det(v_1,-)$} &                               &                            &                           &    \\ \cline{1-2}
\multicolumn{1}{|c|}{}   & \multicolumn{1}{c|}{$\det(v_2,-)$}  & $\ddots$                      &                           &    \\ \cline{2-2} \cline{4-4}
                         &                         & \multicolumn{1}{c|}{$\ddots$} & \multicolumn{1}{c|}{$\det(v_{p-1},-)$} &    \\ \cline{4-5} 
                         &                             &                            & \multicolumn{1}{c|}{}     & $\det(v_p,-)$ \\ \hline
\end{tabular}
\end{center}
The columns correspond to the vertices $V_1,\dots,V_p$, the top rows to the edges of the disorienting cycle, and the bottom rows to the evaluation of the fixed ends. As the cycle is disorienting, it is not possible to trivialize the lattice along the cycle. This is responsible for the two positive entries in the last column, leading to a non-zero determinant. For the copy of $\ZZ^2$ corresponding to $V_i$, we can take a basis of the form $(v_i,v'_i)$ such that the block $\det(v_i,-)$ becomes $(0\ 1)$. We then expand with respect to the bottom rows:
\begin{center}
\begin{tabular}{|ccccc|}
\hline
\multicolumn{1}{|c|}{$\det(u_1,v_1)$} &                         &                                  & \multicolumn{1}{c|}{}     & $\det(u_1,v_p)$ \\ \cline{2-2} \cline{5-5} 
\multicolumn{1}{|c|}{$-\det(u_2,v_1)$} & \multicolumn{1}{c|}{$\det(u_2,v_2)$} &                                   &                           &    \\ \cline{1-1}
\multicolumn{1}{|c|}{}   & \multicolumn{1}{c|}{$-\det(u_3,v_2)$} & $\ddots$                       &                           &    \\ \cline{2-2} \cline{4-4}
                         &                         & \multicolumn{1}{c|}{$\ddots$} & \multicolumn{1}{c|}{$\det(u_{p-1},v_{p-1})$} &    \\ \cline{5-5} 
                         &                             & \multicolumn{1}{c|}{}      & \multicolumn{1}{c|}{$-\det(u_p,v_{p-1})$}   & $\det(u_p,v_p)$ \\ \hline
\end{tabular}
\end{center}
For each column, as $m_{V_i}=\det(u_i,v_i)=\det(u_{i+1},v_i)$, we can factor the vertex multiplicity $m_{V_i}$. Finally, we are left with computing the following determinant:
$$\begin{vmatrix}
1 & & &  & 1 \\
-1 & 1 &  & &  \\
 & -1 & \ddots  & &  \\
 & & \ddots & 1 &  \\
 & & &  -1 & 1 \\
\end{vmatrix} = 2.$$\qedhere
\end{proof}

\begin{rem}
Note that the multiplicity of the curve depends on the position of the marked points and not just on its combinatorial type. This is similar to the multiplicity of a tropical curve in a linear system on an abelian surface introduced in \cite{blomme2022abelian2}.
\end{rem}

We can thus consider the count
$$N^\delta_{g,aE+bF}(\P)=\sum_{\Gamma\supset\P}m_\Gamma^\CC,$$
of tropical curves of genus $g$ in the class $aE+bF$ in the M\"obius strip $\TT M_\delta$ passing through $\P$. Alternatively, we can consider the relative counts
$$N^\delta_{g,aE+bF}(\mu,\nu)(\P).$$

We can define refined multiplicities (see \cite{block2016refined}) for curves on $\TT M_\delta$:

\begin{defi}
Let $h:\Gamma\to\TT M_\delta$ be a parametrized tropical curve of genus $g$ in the class $aE+bF$, tangency profile $\mu+\nu\vdash 2b$ meeting the constraints. We set
$$m_\Gamma^q=\frac{1}{I_q^\mu}2^k\prod_V [m_V]_q,$$
where $k$ is the number of disorienting cycles in the complement of the marked points,  $[a]_q=\frac{q^{a/2}-q^{-a/2}}{q^{1/2}-q^{-1/2}}$ denotes the $q$-analog, and $I_q^\mu=\prod [i]_q^{\mu_i}$.
\end{defi}

We then consider the refined counts
$$BG^\delta_{g,aE+bF}(\P)=\sum_{\Gamma\supset\P}m_\Gamma^q,$$
along with the relative refined count $BG^\delta_{g,aE+bF}(\mu,\nu)(\P)$.

\subsection{Invariance statement}

The invariance for the classical count is a consequence of the correspondence statement. As refined multiplicities specialize on classical multiplicities when $q$ goes to $1$, it is also a consequence of the invariance statement for the refined counts we show below.

\begin{theo}\label{theorem-refined-invariance}
The counts $BG^\delta_{g,aE+bF}(\P)$ and $BG^\delta_{g,aE+bF}(\mu,\nu)(\P)$ do not depend on the choice of the constraints as long as they are generic.
\end{theo}

\begin{proof}
We proceed as in the proofs of the invariant counts provided in \cite{blomme2021floor,blomme2022abelian1,blomme2022abelian2,itenberg2013block} and reduce to the planar case in \cite{itenberg2013block}. Let $(\P_t)_{t\in [0;1]}$ be a generic path between two generic choices of constraints $\P_0$ and $\P_1$ in $\TT M_\delta$. We assume that the constraints move one at a time: either a unique point or the position of a fixed end moves. Let $h:\Gamma\to\TT M_\delta$ be a simple solution. By Proposition \ref{proposition shape of solution}, when removing all marked points and fixed ends from $\Gamma$ except the moving constraint, we are in one of the following situations:
	\begin{itemize}[label=$\ast$,noitemsep]
	\item we have a unique orienting cycle,
	\item we have a path linking two unfixed ends,
	\item we have a component with two disorienting cycles,
	\item we have a component with a disorienting cycle and an unfixed end.
	\end{itemize}

When slightly moving the constraint, the solutions slightly deform accordingly. As long as no edge length goes to $0$, both combinatorial type and multiplicity of the curves remain the same, hence we have local invariance. It remains to check what happens when an edge length goes to $0$, i.e. $\P_t$ becomes non-generic for some value of $t$. We call this \emph{crossing a wall}. 

\smallskip

Assume $\P_{t_\ast}$ is non-generic. Hence, there is a solution with at least a quadrivalent vertex, obtained by deformation of a simple solution. Moreover, for every $t\in ]t_\ast-\varepsilon;t_\ast+\varepsilon[$ with $t\neq t_\ast$, $\P_t$ is generic again and the solutions are simple. Let $h:\Gamma\to\TT M_\delta$ be a simple solution near the wall. To reduce to the planar setting, we use the cutting procedure from Remark \ref{remark induced menelaus}. As there is a minimal length for non-contractible cycles, we can choose an admissible set $\Q\subset\Gamma$ such that none of the cut edges has a length that vanishes under deformation. Thus, we are left with deformation of tropical curves inside $\RR^2$. The choice of $\Q$ ensures that for $t\in]t_\ast-\varepsilon;t_\ast+\varepsilon[$, $\Q$ deforms and keeps being admissible through the deformation. According to \cite{itenberg2013block}, the following can occur during the deformation:
\begin{itemize}[label=-,noitemsep]
\item a rectangular shaped cycle with a marked point gets contracted to a pair of quadrivalent vertices linked by a pair of parallel edges,
\item a quadrivalent vertex appears,
\item a marked point merges with a vertex.
\end{itemize}

The first two cases are treated as in \cite{itenberg2013block}. In the first case there is one curve on each side of the wall with the same multiplicity, as the marked point just jumps from one side of the cycle to the other. This is depicted in Figure \ref{fig-walls-as-usual}(a). For quadrivalent vertices, there are three adjacent combinatorial types, and the invariance is already proved in \cite{blomme2021floor,itenberg2013block}, and depicted in Figure \ref{fig-walls-as-usual}(b)(b').

\begin{figure}[h]
\begin{center}
\begin{tabular}{ccc}
\begin{tikzpicture}[x=0.5cm, y=0.5cm,scale = 0.9]
\draw[line width=1pt] (-2,-2)--++(2,2)--++(3,0)--++(2,-2);
\draw[line width=1pt] (-2,6)--++(2,-2)--++(3,0)--++(2,2);
\draw[line width=1pt] (0,0)--++(0,4);
\draw[line width=1pt] (3,0)--++(0,4);

\draw[line width=1pt,dotted] (0,0)--++(1,1)--++(1,0)--++(1,-1);
\draw[line width=1pt,dotted] (0,4)--++(1,-1)--++(1,0)--++(1,1);
\draw[line width=1pt,dotted] (1,1)--++(0,2);
\draw[line width=1pt,dotted] (2,1)--++(0,2);

\draw[line width=0.5pt,red,->] (3.5,1.5)--++(-3,0);
\draw[dashed,purple] (1.5,-2)--++(0,8);
\draw[red] (1,2) node {$\times$};
\draw[red] (2,2) node {$\times$};
\draw[red] (3,2) node {$\times$};
\end{tikzpicture} & \begin{tikzpicture}[x=0.5cm, y=0.5cm,scale = 0.9]]
\draw[line width=2pt] (-3,0)--(0,0)--(0,-3) (0,0)--(1,1)--(5,-1) (1,1)--(-1,5);
\draw[line width=1.5pt,blue] (-3,0)--(3,0)--(5,-1) (0,-3)--(0,3)--(-1,5) (0,3)--(3,0);

\draw[red,->] (-1,4)--(-2,3);
\draw[dashed,purple] (-2.5,5)--(1.5,-3);
\draw[line width=0.5pt] (-3,0)--(5,0);
\draw[line width=0.5pt] (0,-3)--(0,5);
\draw[line width=0.5pt,red] (-1,5)--(3,-3);
\draw[line width=0.5pt,red,dashed] (-3,4)--(0.5,-3);
\end{tikzpicture} & \begin{tikzpicture}[x=0.5cm, y=0.5cm,scale = 0.9]]
\draw[line width=2pt] (-3,0)--(-1,0)--(0,-1)--(0,-3) (0,-1)--(4,-2) (-1,0)--(-3,4);

\draw[red,->] (-1,4)--(-2,3);
\draw[dashed,purple] (-2.5,5)--(1.5,-3);
\draw[line width=0.5pt] (-3,0)--(5,0);
\draw[line width=0.5pt] (0,-3)--(0,5);
\draw[line width=0.5pt,red,dashed] (-1,5)--(3,-3);
\draw[line width=0.5pt,red] (-3,4)--(0.5,-3);
\end{tikzpicture} \\
(a) & (b) & (b') \\
\end{tabular}
\end{center}
\caption{\label{fig-walls-as-usual}In (a), a cycle gets contracted by moving the marked point. Deforming further, the cycle opens again and the marked point has changed side. In (b) and (b'), the movement of the red line forces the solution to pass through a quadrivalent vertex. If the red line passes into the upper right quadrant, we have two tropical solutions depicted in (b) in black and blue. If not, we have one tropical solution depicted in (b').}
\end{figure}
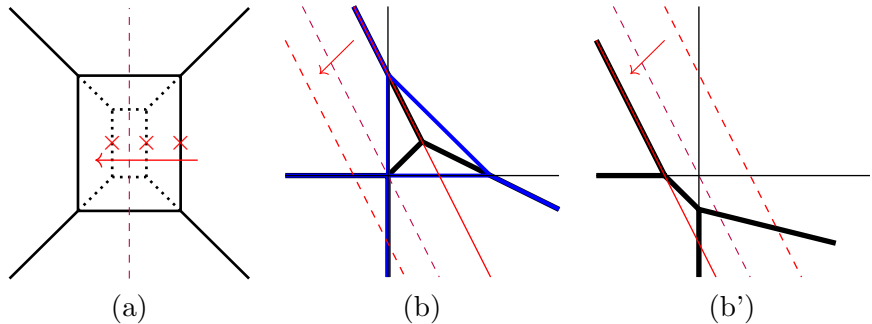

We are left with the case where a marked point meets a vertex during the deformation. In the classical case, there are only two adjacent combinatorial types that provide solutions since the complement of marked points is a forest (a set of rooted trees). Here, due to the different description of solutions given in Proposition \ref{proposition shape of solution}, we have more cases, just as in \cite{blomme2022abelian2}.

Let $V$ be the vertex and $P$ the marked point that meet through the deformation. Let $a$, $b$ and $c$ be the three edges adjacent to $V$ and $\AAA,\BBB$ and $\CCC$ be the connected components of $\Gamma-\left( h^{-1}(\P\backslash\{P\})\cup \{V\} \right)$. Some of these components may coincide if $V$ lies on a cycle.

By Proposition \ref{proposition shape of solution}, $\AAA \cup \BBB \cup \CCC \cup\{V\}$ contains exactly two free ends, two disorienting cycles or one of each, and the marked point $P$ separates the two (or lies on the orienting cycle in case the two disorienting cycles have non-disjoint support).

\begin{figure}[h]
\begin{center}
\begin{tabular}{ccc}
\begin{tikzpicture}[x=0.5cm, y=0.5cm,scale = 0.9]]
\draw (0,0) node[below right] {$V$};
\draw (0,0) node {$\bullet$};
\draw[red] (-4,0)--(0,0) node[midway,above] {$b$};
\draw[red] (-4,0) node[left] {$\BBB$};
\draw[green] (0,-4)--(0,0) node[midway,left] {$c$};
\draw[green] (0,-4) node[below] {$\CCC$};
\draw[blue] (3,3)--(0,0) node[midway,below right] {$a$};
\draw[blue] (3,3) node[above right] {$\AAA$};
\draw (0,0) node {$\bullet$};
\end{tikzpicture} & \begin{tikzpicture}[x=0.5cm, y=0.5cm,scale = 0.9]]
\draw (-4,0)--(0,0)--(0,-4);
\draw[dotted] (-4,-0.5)--(-0.5,-0.5)--(-0.5,-4);
\draw[dotted] (-4,-1)--(-1,-1)--(-1,-4);
\draw (0,0)--(3,3);
\draw[dotted] (-1,-1)--(0,0);
\draw[red] (-2,0) node {$\times$};
\draw[red] (-1.5,-0.5) node {$\times$};
\draw[red] (-1,-1) node {$\times$};
\draw[red] (0,-2) node {$\times$};
\draw[red] (-0.5,-1.5) node {$\times$};
\draw[line width=0.5pt,red,->] (-2.5,-0.5)--++(2,-2);
\draw[dashed,purple] (-3,-3)--(3,3);
\end{tikzpicture} & \begin{tikzpicture}[x=0.5cm, y=0.5cm,scale = 0.9]]
\clip (-4,-4) rectangle (3,3);
\draw (0,0)--++(-6,0)++(6,0)--++(0,-6)++(0,6)--++(6,6);
\draw[dotted,blue] (-0.5,0.5)--++(-6,0)++(6,0)--++(0,-6)++(0,6)--++(6,6);
\draw[blue] (-1,1)--++(-6,0)++(6,0)--++(0,-6)++(0,6)--++(6,6);
\draw[dotted,green] (0.5,-0.5)--++(-6,0)++(6,0)--++(0,-6)++(0,6)--++(6,6);
\draw[green] (1,-1)--++(-6,0)++(6,0)--++(0,-6)++(0,6)--++(6,6);
\draw (0,0)--(3,3);
\draw[dotted] (-1,-1)--(0,0);
\draw[red] (-1,-1) node {$\times$};
\draw[red] (-0.5,-0.5) node {$\times$};
\draw[red] (1,1) node {$+$};
\draw[line width=0.5pt,red,->] (0.5,1.5)--++(-2,-2);
\draw[dashed,purple] (-3,3)--(3,-3);

\end{tikzpicture} \\
(a) & (b) & (c) \\
\end{tabular}
\caption{\label{fig-walls-disorienting-cycle}In (a) the three components cut by the removal of vertex $V$. In (b), the wall where only two out of the three adjacent combinatorial types are allowed. The marked point moves from one edge to the other, deforming the curve. In (c), the wall where all three adjacent combinatorial types are allowed. When the marked point crosses the wall, the black tropical curve becomes the pair of curves given in blue and green. The black curve has an additional disorienting cycle which gets cut on the other side of the wall.}
\end{center}
\end{figure}
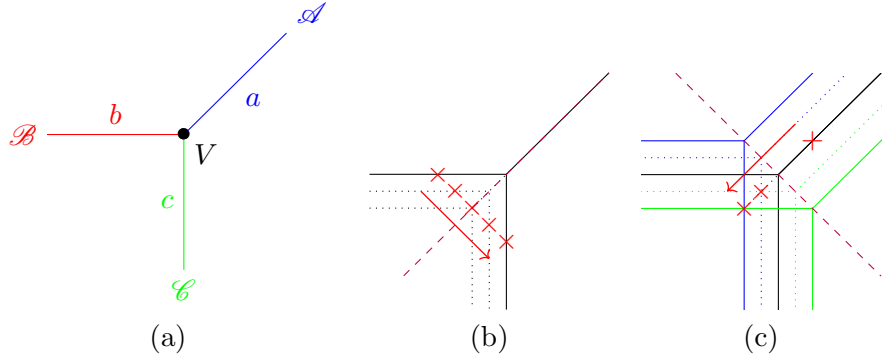

\begin{enumerate}[label=(\roman*)]
\item Assume that all three components are distinct. By the discussion above, exactly one component is bounded, say $\AAA$, while the other components contain either a free end or a disorienting cycle. Then $P$ must lie on either $b$ or $c$. Thus, degeneration leads to two different combinatorial types with the same multiplicity, depending on which side of the line spanned by $a$ the point $P$ lies. This corresponds to Figure \ref{fig-walls-disorienting-cycle}(b).

\item Now, assume that $\BBB=\CCC$. Then, $\BBB$ has at least one cycle $\gamma$ passing through $b$ and $c$. If $\gamma$ is orienting, $P$ needs to lie on $b$ or $c$ and we conclude as in (i). If $\gamma$ is disorienting, there are three combinatorial types.
The first corresponds to $P \in a$, and the other to $P\in b$ or $c$. If $P$ lies on either $b$ or $c$, the disorientng cycle $\gamma$ is now cut by a marked point. Since $\BBB = \CCC$, both combinatorial types this contributes lie on the other side of the wall, and the invariance of multiplicities follows by
$$2^{k-1}\prod m_V+2^{k-1}\prod m_V=2^{k}\prod m_V.$$
This case is depicted on Figure \ref{fig-walls-disorienting-cycle}(c).

\item Finally, assume $\AAA=\BBB=\CCC$. For each pair of edges in $a,b,c$ there exists a cycle $\gamma_{ij}$ passing through both. We have $\gamma_{ab}=\gamma_{ac}+\gamma_{bc}$ in $H_1(\Gamma,\ZZ/2\ZZ)$. Taking the image in $H_1(\TT M_\delta,\ZZ/2\ZZ)$, we see that at least one cycle is orienting. By Proposition \ref{proposition shape of solution}, at most one cycle is orienting, thus exactly one, say $\gamma_{bc}$ is orienting. Hence, $P$ needs to lie on $\gamma_{bc}$, i.e. in  $b$ or $c$, leading to two adjacent combinatorial types, and we conclude as in (i).
\end{enumerate}

\end{proof}

\begin{rem}
It would be interesting to see if the factor $2^k$ appearing in the complex multiplicity can also be refined.
\end{rem}

\section{Correspondence statements}

\subsection{Complex setting}\label{sec-complex-setting}
In the following, we develop the complex setting starting with the definition of the two ruled surfaces $\CC M_0$ and $\CC M_1$ by mimicking the construction of the tropical M\"obius strip. We then identify these surfaces in the classification provided by \cite[Theorem 2.12 and 2.15]{hartshorne2013algebraic}.

	\subsubsection{The surface $\CC M_0$.} Consider the trivial plane bundle $\tau$ over $\CC^*$ and the swapping endomorphism
	$$\Phi_0:(z,u_0,u_1)\longmapsto (\lambda z,u_1,u_0)\in\CC^*\times\CC^2.$$
	The quotient of $\tau$ by $\Phi_0$ descends to a plane bundle $\EEE_0$ over $\CC E=\CC^*/\langle\lambda\rangle$. Its projectivization $\PP(\EEE_0)$ is the ruled surface $\CC M_0$. It can be equivalently constructed as the quotient of $\CC^*\times\CC P^1$ by the action
    $$\varphi_0(z,[u_0:u_1])=(\lambda z,[u_1:u_0]).$$
    The dense torus $(\CC^*)^2\subset\CC^*\times\CC P^1$ is stable under the action which restricts to
    $$\phi_0:(z,w)\mapsto \left(\lambda z,\frac{1}{w}\right),$$
    providing the complex counterpart to the tropical picture. The ruled surface $\CC M_0$ is thus a compactification of the quotient of $(\CC^*)^2$ by the action of $\phi_0$.
    
    Taking the quotient by $\Phi_0^2$ instead (resp. $\varphi_0^2,\phi_0^2$), we obtain a plane (resp. $\CC P^1$, $\CC^*$) bundle over the elliptic curve $\CC\widetilde{E}=\CC^*/\langle\lambda^2\rangle$. Projectivization yields a ruled surface $\CC C_0$ which is a two-to-one cover of $\CC M_0$. The surface $\CC C_0$  is the total space of the trivial $\CC P^1$-bundle, and the complex analogue of the tropical cylinder $\TT C_0$ \cite{blomme2021floor}.
    
    \smallskip
    
    To finish, we need to identify $\CC M_0$ in the classification of ruled surfaces over an elliptic curve. The plane bundle $\EEE_0$ splits into a sum of two line bundles. Indeed, we can consider the sections $s_\pm:z\in\CC^*\mapsto (1,\pm 1)\in\CC^2$ of the trivial bundle. Both are non-vanishing. Hence, each defines a line bundle inside the trivial plane bundle and both are stable under the action of $\Phi_0$. Thus, both sections induce line bundles over $\CC E$, yielding two non-intersecting sections of $\CC M_0$. Thus, the plane bundle $\EEE_0$ is split, and $\CC M_0$ is the projective completion of a line bundle over $\CC E$.
    
    By a change of coordinates $w'=\frac{w-1}{w+1}$, $\CC M_0$ can also be written as the quotient of $\CC^*\times\CC P^1$ by the action of
    $$\varphi'_0:(z,w')\longmapsto (\lambda z,-w'),$$
    which is the projective completion of a $2$-torsion line bundle over $\CC E$.

    \subsubsection{The surface $\CC M_1$.}
    
    We now construct the surface $\CC M_1$ in an analogous way, considering the trivial plane bundle over $\CC^*$ with the twisted swapping action
	$$\Phi_1:(z,u_0,u_1)\longmapsto (\lambda z,zu_1,u_0)\in\CC^*\times\CC^2.$$
	The quotient by the action yields a plane bundle $\EEE_1$ over $\CC E$. Its projectivization $\PP(\EEE_1)$ is the ruled surface $\CC M_1$. Projectivizing first, we also obtain $\CC M_1$ as the quotient of $\CC^*\times\CC P^1$ by the action
	$$\varphi_1(z,[u_0:u_1])=\left(\lambda z,[zu_1:u_0]\right).$$
	The action on the dense torus is given by
	$$\phi_1:(z,w)\mapsto\left(\lambda z,\frac{z}{w}\right),$$
	which is again the complex counterpart to the tropical picture. We have $\phi_1^2(z,w)=(\lambda^2z,\lambda w)$, and the quotient $\CC C_1=\CC^*\times\CC P^1/\langle\phi_1^2\rangle$ is the complex counterpart to the tropical cylinder $\TT C_1$. It is the total space of a $2$-torsion line bundle over $\CC\widetilde{E}=\CC^*/\langle\lambda^2\rangle$.
	
	\smallskip
	
    The surface $\CC M_1$ is a ruled surface if it has a section. Sections of $\CC M_1\to\CC E$ are sections of $\CC C_1\to\CC\widetilde{E}$ invariant under the action induced by $\varphi_1$. As the latter is the projective completion of a $2$-torsion line bundle, these are the two sections provided by the $0$ and $\infty$ section. Unfortunately, none of them is invariant under the action induced by $\varphi_1$: they are exchanged by the action and form a multisection of $\CC M_1\to\CC E$. To find sections, we have to investigate further.
	
	Take  the meromorphic function $\theta:\CC^*\to\CC$ given by $\theta(z)=\sum_{-\infty}^\infty \lambda^{n^2}z^n$. It is the $\theta$-function (see \cite{griffiths2014principles}) on $\CC\widetilde{E}=\CC^*/\gen{\lambda^2}$, and it satisfies
	$$\theta(\lambda^2 z)=\frac{1}{\lambda z}\theta(z) \text{ and }\theta\left(\frac{1}{z}\right)=\theta(z).$$
	It is the only meromorphic function satisfying the first equation up to multiplication by a scalar. Using both equations, we can see that $\theta(-\lambda)=-\theta(-\lambda)$, thus $\theta(-\lambda)=0$. Moreover, $-\lambda$ is  the only zero of $\theta$ modulo multiplication by $\lambda^2$. Any quotient $f(z)=\frac{\theta(\alpha z)}{\theta(\alpha \lambda z)}$ gives a section of $\CC C_1$ since it satisfies
	$$f(\lambda^2 z)= \frac{\lambda\cdot\lambda\alpha z}{\lambda\cdot\alpha z}\frac{\theta(\alpha z)}{\theta(\lambda\alpha z)}=\lambda f(z).$$
	Moreover,
	$$\phi_1\left(z,\mu\frac{\theta(\alpha z)}{\theta(\lambda\alpha z)}\right)=\left(\lambda z,\frac{z}{\mu}\frac{\theta(\alpha\lambda z)}{\theta(\alpha z)}\right)=\left(\lambda z,\frac{z}{\mu\lambda\alpha z}\frac{\theta(\alpha\cdot\lambda z)}{\theta(\lambda\alpha\cdot\lambda z)}\right).$$
	Thus, if $\mu^2\lambda\alpha=1$, it gives a section of $\CC M_1\to\CC E$ and $\CC M_1$ is a ruled surface over $\CC E$.
	
	Sections of $\CC M_1$ pull back to sections of $\CC C_1$. These sections always intersect since we excluded the $0$ and $\infty$-sections. Thus, we cannot find disjoint sections of $\CC M_1$, which prevents $\EEE_1$ from being split. As the line bundle $\Lambda^2\EEE_1$ on $\CC E$ is of odd degree, the classification given by \cite{hartshorne2013algebraic} ensures that $\CC M_1$ is the unique non-split ruled surface of degree $1$, given as $\PP(\mathscr{E})$ where $\mathscr{E}$ is a non-split vector bundle of rank 2 over an elliptic curve fitting in a short exact sequence 
	$$0\to\O\to\mathscr{E}\to \O(p)\to 0.$$

	\subsubsection{Curves and divisors.} As ruled surfaces, the Picard group of the $\CC M_{\delta}$ are 
	$$\mathrm{Pic}(\CC M_{\delta})= \ZZ \oplus \pi^{*} \text{Pic}(\CC E).$$
	A basis of the second homology group is provided by the class of a section and the class of a fiber $F$. We prefer to replace the class of a section by the class of the boundary divisor $E$, which is a $2$-section. As $E\cdot F=2$, they do not form a basis of the $H_2(\CC M_\delta,\ZZ)$. Computing the intersection number of a section with $E$ and $F$, we see that $H_2(\CC M_\delta,\ZZ)$ has the same description as the tropical homology group, i.e.
	$$H_2(\CC M_\delta,\ZZ)=\left\{ aE+bF|a,b\in\frac{1}{2}\ZZ,2b\equiv 2a\delta \text{ mod }2  \right\}.$$

\begin{defi}
    Let $\CC \Gamma_g$ be a genus $g$ Riemann surface. We say that a curve $\varphi:\CC \Gamma_g \rightarrow \CC M_\delta$  is of bidegree $(a,b)$ on a M\"obius strip $\CC M_\delta$ if $\varphi_\ast([\CC \Gamma_g])=aE+bF\in H_2(\CC M_\delta,\ZZ)$.
\end{defi}

	Note that the boundary divisor is not anticanonical: due to the non-orientability of the M\"obius strip, the meromorphic $2$-form $\frac{\mathrm{d}z}{z}\wedge\frac{\mathrm{d}w}{w}$ on $(\CC^*)^2$ does not induce a $2$-form on the quotient $\CC M_\delta$. However, it is possible to construct such a $2$-form using $\theta$-functions, proving the boundary divisor is numerically equivalent to the canonical class.
	
	In the case of $\CC M_0$ ($\CC M_1$ is treated similarly), we have the meromorphic $2$-form $\Omega_{\CC C_0}=\frac{\mathrm{d}z}{z}\wedge\frac{\mathrm{d}w}{w}$ on $\CC^*/\gen{\lambda^2}\times\CC P^1$, which satisfies $\varphi_0^*\Omega_{\CC C_0}=-\Omega_{\CC C_0}$. Now, let $\lambda^{1/2}$ be a square root of $\lambda$ and $\theta(z)=\sum (\lambda^{1/2})^{n^2}z^n$ be the associated $\theta$-function on $\CC E$, which satisfies $\theta(\lambda z)=\frac{1}{\lambda^{1/2}z}\theta(z)$.	Thus, the function $f(z)=\frac{\theta(z)}{\theta(-z)}$ satisfies $f(\lambda z)=-f(z)$, and $f(\lambda^2 z)=f(z)$ and descends to a meromorphic function on $\CC\widetilde{E}$. Thus, the $2$-form $f(z)\Omega_{\CC C_0}$ is $\varphi_0$-invariant and descends to a meromorphic $2$-form on $\CC M_0$. It has poles along the boundary divisor and along a fiber, and zeros along one of the fibers. The poles and zeros along the fibers correspond to poles and zeros of $f$.

\subsection{Log-invariants}

In this section, we briefly define the log-GW invariants of the surfaces $\CC M_\delta$ as the complex counterpart to the tropical problem, following \cite{cavalieri2021counting}, and refer to \cite{cavalieri2021counting,gross2013logarithmic} for more details.

We consider one of the surfaces $\CC M_\delta$ and fix the following discrete data $\beta_{a,\mu,\nu}$:
\begin{itemize}[label=$\ast$,noitemsep]
\item positive integers $g,n$ and a half-integer $a$,
\item two partitions $\mu$ and $\nu$. We set $2b=\|\mu\|+\|\nu\|$ and assume that $2b\equiv 2a\delta \text{ mod }2$.
\end{itemize}

Log-stable maps and Log-GW invariants were introduced in \cite{abramovich2014stable,chen2014stable,gross2013logarithmic}. There is a \textit{moduli stack of log-stable maps} with logarithmic structure
$$\M^{\log}_{g,n+|\mu|+|\nu|}(\CC M_\delta,\beta_{a,\mu,\nu}),$$
parametrizing log-stable maps from a source of genus $g$ with $n+|\mu|+|\nu|$ marked points to $\CC M_\delta$. The partitions $\mu$ and $\nu$ impose the intersection profile with the boundary divisor of  $\CC M_\delta$.

For each of the $n$ marked points, the moduli stack is equipped with an \textit{evaluation map}
$$\mathrm{ev}_i:\M^{\log}_{g,n+|\mu|+|\nu|}(\CC M_\delta,\beta_{a,\mu,\nu})\longrightarrow \CC M_\delta.$$
It maps a log-stable map to the image of the $i$th marked point. Moreover, for each of the marked points mapped to the boundary divisor, we have an additional evaluation map with values in the corresponding divisor:
$$\widehat{\mathrm{ev}_i}:\M^{\log}_{g,n+|\mu|+|\nu|}(\CC M_\delta,\beta_{a,\mu,\nu})\longrightarrow \CC\widetilde{E}.$$

The moduli stack $\M^{\log}_{g,n+|\mu|+|\nu|}(\CC M_\delta,\beta_{a,\mu,\nu})$ is  a proper Deligne-Mumford stack equipped with a \textit{virtual fundamental class} $[\M]^{\log}$ in degree $g-1+|\mu|+|\nu|+n$. We define the log-GW invariants by intersecting the virtual fundamental class with classes provided by the evaluation morphisms. For the dimension counts to agree, we take $n=|\nu|+g-1.$

\begin{defi}
The stationary log-GW invariant is defined by
$$\N_{g,aE+bF}^\delta(\mu,\nu)=\int_{[\M]^{\log}} \prod_1^n\mathrm{ev}_i^*(\mathrm{pt})\prod_1^{|\mu|}\widehat{\mathrm{ev}_i}^*(\mathrm{pt}),$$
where $\mathrm{pt}$ is the cohomology class Poincar\'e dual to a point.
\end{defi}

\subsection{Correspondence}

We now relate the log-GW invariants to the tropical counts. It is possible to construct a log-smooth family of the ruled surfaces $\CC M_\delta$ by varying the base elliptic curve, with a central fiber that is a union of toric surfaces. Log-GW invariants are constant in families, and the decomposition formula by Abramovich-Chen-Gross-Siebert \cite{abramovich2020decomposition} relates the log-GW invariants of the central fiber to tropical invariants.

\begin{theo}\label{theorem-corresp}
The tropical invariant and the log-GW invariant agree:
$$\N^\delta_{g,aE+bF}(\mu,\nu)=N^\delta_{g,aE+bF}(\mu,\nu).$$
\end{theo}

\begin{proof}
To apply the Abramovich-Chen-Gross-Siebert decomposition formula \cite{abramovich2020decomposition}, one needs a family of surfaces. We construct this family using the construction of the surfaces $\CC M_\delta$ presented in Section \ref{sec-complex-setting} and the tropical curves solving the tropical enumerative problem described in Section \ref{sec-tropical-mobius-strips}.

\smallskip

\textbf{Step 1: Constructing a subdivision.} Consider one of the tropical M\"obius strips $\TT M_\delta$, along with a generic configuration $\P$ of $|\nu|+g-1$ points and $|\mu|$ points on the boundary divisor. We can assume that they have rational coordinates. As the configuration is generic, there is a finite number of tropical curves in the class $aE+bF$ of tangency profile $\mu+\nu$ and matching the point and tangency constraints. We then consider a polyhedral subdivision $\Xi$ of $\TT M_\delta$ such that each tropical curve factors through the $1$-skeleton of the subdivision. We can always take such a subdivision by taking the common refinement of subdivisions corresponding to each tropical curve. Up to scaling, we can assume that the coordinates of the vertices of the subdivision along with the length $l$ of the tropical elliptic curve $\TT E$ are integers. Moreover, $\Xi$ projects onto a subdivision $\Sigma$ of $\TT E$.

\smallskip

We unfold the polyhedral subdivision $\Xi$ of $\TT M_\delta$ to get a polyhedral subdivision $\widetilde{\Xi}$ of its universal cover $\RR^2$ by taking the preimage. We do the same for $\Sigma$ to get a polyhedral subdivision $\widetilde{\Sigma}$ of $\RR$. By construction, $\widetilde{\Sigma}$ is stable by translation by $l$ and $\widetilde{\Xi}$ is stable under the action of $\varphi_\delta$, which lifts the translation by $l$ in $\RR$.

\smallskip

\textbf{Step 2: Constructing the family.} We consider the cone over the polyhedral subdivisions $\widetilde{\Sigma}\times\{1\}\subset\RR^2$ and $\widetilde{\Xi}\times\{1\}\subset\RR^3$. This yields two (infinite) fans endowed with a map to $\RR_{\geqslant 0}$ provided by the projection on the last coordinate. Using the construction of toric varieties for these (infinite) fans, we get two complex manifolds $\CC\widetilde{\EEE}$ and $\CC\widetilde{\MMM}$ with a map $\CC\widetilde{\MMM}\to\CC\widetilde{\EEE}$, and a map to $\CC$. The threefold $\CC\widetilde{\MMM}$ is a partial compactification of $(\CC^*)^3$ while $\CC\widetilde{\EEE}$ is a partial compactification of $(\CC^*)^2$.

\smallskip

The translation action on $\widetilde{\Sigma}$ lifts to the fan by $(x,\tau)\mapsto (x+\tau l,\tau)$. It induces a map on the complex surface $\CC\widetilde{\EEE}$ that extends the map of the dense torus
$$(z,t)\longmapsto (\lambda t^l z,t).$$
Similarly, the action of $\varphi_\delta$ lifts to the fan by $(x,y,\tau)\mapsto (x+\tau l,\delta z-w,\tau)$, and there is an extension to the threefold $\CC\widetilde{\MMM}$ that extends the map of $(\CC^*)^3$ to itself
$$(z,w,t)\longmapsto \left(\lambda t^l z,\frac{z^\delta}{w},t\right).$$
We can then consider the quotient by the above actions to get manifolds $\CC\EEE$ and $\CC\MMM$, along with a map $\CC\MMM\to\CC\EEE$ and maps to $\CC$.

\smallskip

The fiber $\CC\MMM_t$ for $t\neq 0$ is the ruled surface $\CC M_\delta$ over the base elliptic curve $\CC\EEE_t=\CC^*/\gen{\lambda t^l}$. The central fiber of the family of elliptic curves $\CC\EEE_0$ is a chain of copies of $\CC P^1$ meeting along their toric divisors, i.e. their respective $0$ and $\infty$. The central fiber $\CC\MMM_0$ is a union of toric surfaces glued along their toric divisors. This construction is just a periodic version of the construction of a family of toric surfaces as done in \cite{nishinou2006toric}.

\smallskip

\textbf{Step 3: Applying the decomposition formula.} The tropical points in $\P$ and the points on the boundary divisor give rise to sections of the family. The log-GW invariants being constant in families, we can compute the log-GW invariant for the central fiber $\CC\MMM_0$. Using \cite[Theorem 5.4]{abramovich2020decomposition}, we can intersect the point constraints with the virtual fundamental class $[\M]^{\log}$ to get a virtual fundamental class $[\M^\P]^{\log}$ of degree $0$, and the log-GW invariant becomes the degree of this $0$-cycle. The virtual fundamental class $[\M^\P]^{\log}$ splits as a sum over the tropical curves solving the tropical enumerative problem:
$$[\M^\P]^{\log}=\sum_{h:\Gamma\to\TT M_\delta} [\M^{h,\P}]^{\log},$$
where $[\M^{h,\P}]^{\log}$ is a virtual fundamental class corresponding to the stable log-maps to the central fiber $\CC\MMM_0$ whose combinatorial type is encoded by $h:\Gamma\to\TT M_\delta$. In other words, the dual graph to the source curve is $\Gamma$, and the component corresponding to a vertex $V$ is mapped to the irreducible component of $\CC\MMM_0$ corresponding to $h(V)$. The log-GW invariant thus splits as a sum
$$\N_{g,aE+bF}^\delta(\mu,\nu)=\sum_{h:\Gamma\to\TT M_\delta}\int_{[\M^{h,\P}]^{\log}}1.$$

\smallskip

\textbf{Step 4: Gluing formula.} As in \cite{cavalieri2021counting}, we are left with the computation of the summands, i.e. the multiplicity $\int_{[\M^{h,\P}]}1$ of each tropical curve. To do so, we use the gluing formula from \cite[Proposition 13]{bousseau2019tropical}, inspired from the proof of \cite[Theorem 1.5,1.6]{kim2018degeneration}, which applies in our setting since it requires a log-smooth family degenerating to a union of toric surfaces meeting along their toric divisors. More precisely, if $V$ is such a vertex, we have a corresponding moduli stack $\M_V$ parametrizing the genus $0$ curves in the toric surface $\CC\MMM_{h(V)}$ (the irreducible component of the central fiber associated to $h(V)$) that have tangency profile with the toric boundary prescribed by the slopes of $h$ on the edges adjacent to $V$ in $\Gamma$. It is endowed with a virtual fundamental class $[\M_V]^{\log}$ and evaluation morphisms. We have a cutting morphism
$$\M^{h,\P}\to\prod_V \M_V,$$
that associates to each log-stable map the restriction to the component associated to the vertex $V$. Moreover, it is mapped to the diagonal $\Delta$ by the evaluation morphism
$$\mathrm{ev}:\prod\M_V\to \prod_e D_e^2,$$
where $D_e$ is the divisor associated to the edge of the subdivision to which the edge $e$ of $\Gamma$ is mapped. The image is called the set of \textit{pre-log} curves. According to \cite{bousseau2019tropical}, this map from $\M^{h,\P}$ to its image is a covering map of degree $\prod_e w_e$ where the product is over the bounded edges of $\Gamma$. One thus needs to count the pre-log curves. Genus $0$ curves to a toric surface with three punctures are parametrized by $(\CC^*)^2$. We thus have the map
$$\prod_V (\CC^*)^2\longrightarrow (\CC^*)^{|E_b|} \times (\CC^*)^{2n}\times(\CC^*)^{|\mu|},$$
which is exactly $\Theta\otimes\CC^*$. As in \cite[Proposition 4.10]{mandel2020descendant}, the set of pre-log curves matching the point constraints is a $\ker\Theta\otimes\CC^*$-torsor, which has thus size the lattice index $|\det\Theta|$ computed in Section \ref{sec-tropical-curves}. In the end, we get the expected tropical multiplicity.
\end{proof}

\begin{rem}\label{rem-lambda-class}
The proof follows the same steps as the proofs in \cite{bousseau2019tropical,cavalieri2021counting}. The difference to \cite{cavalieri2021counting} is that we do not have $\psi$-constraints, which enables us to get the lattice index in the end. The difference with \cite{bousseau2019tropical} is that we do not consider $\lambda$-classes insertions. In particular, the above proof is a particular case of the proof of \cite{bousseau2019tropical}, to which we refer.

In \cite{bousseau2019tropical}, the decomposition formula is used to compute the log-GW invariants with the insertion of a $\lambda$-class, relating them to refined invariants  \cite{block2016refined}. This can also be applied in our setting to prove an analogous result: with $n=2b+g_0-1$, substituting $q=e^{iu}$, one has
$$\left((-i)(q^{1/2}-q^{-1/2})\right)^{2b+2g_0-2}BG^\delta_{g_0,aE+bF}=\sum_{g\geqslant g_0}u^{2g-2+2b}\int_{[\M_{g,n}(\CC M_\delta,aE+bF)]^{\log}}\lambda_{g-g_0}\prod_1^n\mathrm{ev}_i^*(\mathrm{pt}).$$
\end{rem}

\section{Floor diagram algorithm}\label{sec-floor-diag-algo}

\subsection{Floor diagrams}

We start with the formal definition and some examples of abstract \textit{floor diagrams}. In the next subsection, we describe their relationship to \textit{floor-decomposed} tropical curves.

\begin{defi}\label{definition diagram}
A \textit{floor diagram} in $\TT M_\delta$ is an oriented graph with infinite outgoing edges and the additional following data:
\begin{enumerate}[label=$(\roman*)$]
\item Vertices are separated into three disjoint sets: \textit{ground floors}, \textit{\'etages}\footnote{The english language apparently lacks a word for floors which are not ground floors.} and \textit{joints}. \textit{Floors} encompass ground floors and \'etages.
\item An edge $e$, also called \textit{elevator}, has an integer weight $w_e\in\NN$.
\item An \'etage $\F$ has a degree $a_\F\in\NN$.
\item A ground floor $\G$ has a degree $a_\G\in\frac{1}{2}\NN$.
\end{enumerate}
The additional data has to satisfy the following conditions:
\begin{enumerate}[label=(\Alph*)]
\item Ground floors have no incoming edges. Moreover, if $\G$ is a ground floor, we have the balancing condition:
$$\sum_{e\ni\G}w_e\equiv \delta 2a_\G \text{ mod }2 = \begin{cases}
    0 & \text{ for } \TT M_0\\
    2a_\G & \text{ for }\TT M_1.
\end{cases} $$
\item \'Etages have zero divergence: the sum of weights of incoming edges is equal to the sum of weights of outgoing edges.
\item Joints have exactly two outgoing edges of the same weight and no incoming edge.
\end{enumerate}
\end{defi}

\begin{rem}
 The degrees of the ground floors are allowed to be half-integers analogous to the coefficients of $aE+bF$ in $H_{1,1}(\TT M_\delta,\ZZ)$. Further, note that unlike classical floor diagrams, vertices are linearly ordered but all ground floors and joints have the same height. 
\end{rem}

\begin{defi}
Given a floor diagram $\Dfk$, we define the following:
\begin{itemize}[label=$\ast$]
\item The \textit{genus} of $\Dfk$ is its genus as a graph plus its number of floors.
\item The diagram $\Dfk$ is said to be in the class $aE+bF$ if
the sum of the weights of infinite outgoing edges is $2b$, and the sum of degrees of the floors (both ground floors and \'etages) is equal to $a$.

\item A diagram in the class $aE+bF$ is of tangency profile $\mu\vdash 2b$ if it has $\mu_i$ ends of weight $i$ for each $i$.
\end{itemize}
\end{defi}

\begin{figure}
\begin{center}
\begin{tabular}{c c c}
\begin{tikzpicture}[scale = 0.9]
    \node[thick, draw=blue, ellipse,
    align=center,blue,minimum height=24pt, minimum width = 12pt] at  (0,0)(0){\footnotesize $\frac{1}{2}$};
    \node[thick, draw=red, ellipse,
    align=center,minimum height=24pt, minimum width = 12pt,red] at  (1.5,0)(1){$1$};
    \draw [->](0) to [out=30,in=150] (1);
    \draw [->] (0) to [out=-30,in=-150] (1);
    \draw [->] (1) to [out=30,in=180](3,0.15);
    \draw [->] (1) to[out=-30,in=180] (3,-0.15);
    \node[below] at (2.7,-0.5){$(a)$};
    \end{tikzpicture}
&
    \begin{tikzpicture}[scale = 0.9]
\usetikzlibrary{shapes}
    \node[thick, draw=blue, ellipse,
    align=center,blue,minimum height=24pt, minimum width = 12pt] at  (0.25,-0.5)(1){\footnotesize $\frac{1}{2}$};
    \node[thick, draw=orange, ellipse,
    align=center,minimum height=15pt, minimum width = 10pt] at  (0.25,0.5)(2){};
    \node[thick, draw=red, ellipse,
    align=center,minimum height=24pt, minimum width = 12pt,red] at  (1.5,0)(3){$1$};
    \draw[->]  (1) to [out =20, in=210] (3);
    \draw[->]  (1) to [out =-10, in=230] (3);
    \draw[->]  (2) to [out =10, in=150] (3);
    \draw[->]  (2) to [out =-20, in=170] (3);
    \draw [->] (3) to [out=25,in=180](3,0.15);
    \draw [->] (3) to[out=-25,in=180] (3,-0.15);
    \draw [->] (3) --(3,0);
    \node [above] at (2.25,0.2){$2$};
    \node [below] at (2.25,-0.5){$(b)$};
    \end{tikzpicture}  & 
    \begin{tikzpicture}[scale = 0.9]
    \node[thick, draw=blue, ellipse,
    align=center,blue,minimum height=24pt, minimum width = 12pt] at  (0.5,0.5)(0){\footnotesize $\frac{1}{2}$};
    \node[thick, draw=orange, ellipse,
    align=center,minimum height=15pt, minimum width = 10pt] at  (0.5,-0.5)(1){};
    \node[thick, draw=red, ellipse,
    align=center,minimum height=24pt, minimum width = 12pt,red] at  (1.75,0)(2){$2$};
    \node[thick, draw=red, ellipse,
    align=center,minimum height=24pt, minimum width = 12pt,red] at  (3,0)(3){$1$};
    \draw [->] (0) to [out=40,in=180] (4.5, 0.4);
    \draw [->] (0) to [out=0,in=150] (2);
    \draw [->] (1) to [out=0,in=-150] (2);
    \draw [->] (1) to [out=-40,in=-140] (3);
    \draw [->] (2) to (3);
    \draw [->] (3) to [out=30,in=180](4.5,0.15);
    \draw [->] (3) to[out=-30,in=180] (4.5,-0.15);
    \node[above] at (2.325,0){$2$};
    \node [below] at (4,-0.15){$2$};
    \node[below] at (3.75,-0.5){$(c)$};
    \end{tikzpicture}\\
\begin{tikzpicture}[scale = 0.9]
    \node[thick, draw=orange, ellipse,
    align=center,minimum height=15pt, minimum width = 10pt] at  (0,0)(0){};
    \node[thick, draw=red, ellipse,
    align=center,minimum height=24pt, minimum width = 12pt,red] at  (1.5,0)(1){$1$};
    \draw [->](0) to [out=30,in=150] (1);
    \draw [->] (0) to [out=-30,in=-150] (1);
    \draw [->] (1) to [out=30,in=180](3,0.15);
    \draw [->] (1) to[out=-30,in=180] (3,-0.15);
    \node[below] at (2.7,-0.5){$(a)'$};
    \end{tikzpicture}& 
    \begin{tikzpicture}[scale = 0.9]
\usetikzlibrary{shapes}
    \node[thick, draw=blue, ellipse,
    align=center,blue,minimum height=24pt, minimum width = 12pt] at  (1.75,-1)(1){\footnotesize $\frac{1}{2}$};
    \node[thick, draw=blue, ellipse,
    align=center,blue,minimum height=24pt, minimum width = 12pt] at  (1.75,1)(0){\footnotesize $\frac{1}{2}$};
    \node[thick, draw=orange, ellipse,
    align=center,minimum height=15pt, minimum width = 10pt] at  (1.75,0)(2){};
    \node[thick, draw=red, ellipse,
    align=center,minimum height=24pt, minimum width = 12pt,red] at  (3,0)(3){$1$};
    \draw[->]  (0) to [out =-20, in=130] (3);
    \draw[->]  (1) to [out =20, in=-130] (3);
    \draw[->]  (2) to [out =30, in=160] (3);
    \draw[->]  (2) to [out =-30, in=-160] (3);
    \draw [->] (3) to [out=25,in=180](4.5,0.15);
    \draw [->] (3) to[out=-25,in=180] (4.5,-0.15);
    \node [above] at (3.75,0.2){$3$};
    \node [below] at (3.75,-0.5){$(b)'$};
    \end{tikzpicture}
    &
    \begin{tikzpicture}[scale = 0.9]
    \node[thick, draw=blue, ellipse,
    align=center,blue,minimum height=24pt, minimum width = 12pt] at  (0.5,0.5)(0){\footnotesize $\frac{1}{2}$};
    \node[thick, draw=orange, ellipse,
    align=center,minimum height=15pt, minimum width = 10pt] at  (0.5,-0.5)(1){};
    \node[thick, draw=red, ellipse,
    align=center,minimum height=24pt, minimum width = 12pt,red] at  (1.75,0)(2){$2$};
    \node[thick, draw=red, ellipse,
    align=center,minimum height=24pt, minimum width = 12pt,red] at  (3,0)(3){$1$};
    \draw [->] (0) to [out=0,in=150] (2);
    \draw [->] (1) to [out=0,in=-150] (2);
    \draw [->] (1) to [out=-40,in=-140] (3);
    \draw [->] (2) to (3);
    \draw [->] (3) to [out=30,in=180](4.5,0.15);
    \draw [->] (3) to[out=-30,in=180] (4.5,-0.15);
    \node[above] at (2.325,0){$2$};
    \node [above] at (3.75,0.2){$2$};
    \node[below] at (3.75,-0.5){$(c)'$};
    \end{tikzpicture}
\end{tabular}
\caption{\label{fig-floor-diagrams-expl} Some examples of floor diagrams in $\TT M_0$ and $\TT M_1$. In the top row, we give examples in  $\TT M_0$ and in the bottom row, we give some in $\TT M_1$.}
\end{center}
\end{figure}
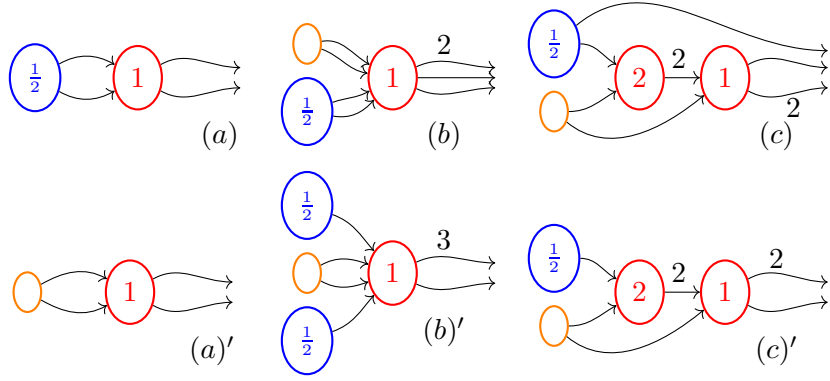

\begin{expl}\label{ex genus class tangency floor diagrams}
In Figure \ref{fig-floor-diagrams-expl}, we give several examples of floor diagrams in $\TT M_0$ and $\TT M_1$. In all pictures in this section, we draw elevators in black, ground floors in blue, \'etages in red and joints in orange. Unlabeled edges have weight one. 
Note that the main difference for floor diagrams for $\TT M_0$ and $\TT M_1$ is the treatment of ground floors. The floor diagrams in Figure \ref{fig-floor-diagrams-expl} are of the following genus, class and tangency profiles:
\begin{center}
\begin{tabular}{c  c | c | c || c c | c | c}
     &  Genus & Class & Tangency & & Genus & Class & Tangency\\
     $(a)$ & $3$ & $\frac{3}{2}E+F$ & $1^2$ & $(a)'$ & $2$ & $E+F$ & $1^2$ \\ 
     $(b)$ & $4$ & $\frac{3}{2}E+2F$ &$1^22^1$& $(b)'$ & $4$ & $2E+2F$ &$3^11^1$\\
     $(c)$ & $4$ & $\frac{7}{2}E+2F$ &$1^22^1$& $(c)'$ & $4$ & $\frac{7}{2}E+\frac{3}{2}F$ &$2^11^1$\\
\end{tabular}
\end{center}\end{expl}

\begin{defi}\label{definition marking diagram}
Let $\Dfk$ be a floor diagram of genus $g$ and tangency profile $\mu$. A \textit{marking} $\mfk$ of a floor diagram $\Dfk$ is an increasing map $\mfk:[\![1;|\mu|+g-1]\!]\to\Dfk$ (i.e. $\mfk(i)\prec\mfk(j)\Rightarrow i<j$, where $\prec$ is the order relation nn the cycle-free oriented graph $\Dfk$), such that:
\begin{enumerate}[label=(\alph*)]
\item No marking is mapped to a ground floor or a joint, and a unique marking is mapped to each \'etage.
\item At most one marking is mapped to an elevator.
\item Each component of the complement of markings on elevators is of one of the following types:
	\begin{enumerate}[label=$(\roman*)$]
	\item It contains a unique ground floor and no cycles or free ends.
	\item It contains a unique free end, and no cycles or ground floors.
	\item It has a unique cycle containing an odd number of joints and no free ends or ground floors.
	\end{enumerate}
\end{enumerate}
\end{defi}

\begin{figure}
    
\begin{center}
\begin{tabular}{c c c c}
    \begin{tikzpicture}[scale = 0.9]
\usetikzlibrary{shapes}
    \node[thick, draw=blue, ellipse,
    align=center,blue,minimum height=24pt, minimum width = 12pt] at  (1.75,0.5)(2){\footnotesize $\frac{1}{2}$};
    \node[thick, draw=orange, ellipse,
    align=center,minimum height=15pt, minimum width = 10pt] at  (1.75,-0.5)(1){};
    \node[thick, draw = red, fill=red, fill opacity =0.5, text opacity = 1, ellipse,
    align=center,minimum height=24pt, minimum width = 12pt,text = black] at  (3,0)(3){$1$};\
    
    \draw[->]  (1) to [out =20, in=210](3);
    \draw[->]  (1) to [out =-10, in=230] (3);
    \draw[->]  (2) to [out =10, in=150] node[midway,fill = black, circle, inner sep = 0, minimum size =5pt]{}  (3);
    \draw[->]  (2) to [out =-20, in=170] node[midway,fill = black, circle, inner sep = 0, minimum size =5pt]{}  (3);
    \draw [->] (3) to [out=25,in=180]node[midway,fill = black, circle, inner sep = 0, minimum size =5pt]{} (4.5,0.15);
    \draw [->] (3) to[out=-25,in=180]node[midway,fill = black, circle, inner sep = 0, minimum size =5pt]{}  (4.5,-0.15);
    \draw [->] (3) --node[midway,fill = black, circle, inner sep = 0, minimum size =5pt]{} (4.5,0);
    \node [above] at (3.75,0.2){$2$};
    \node [below] at (3.75,-0.5){$(b)$};
    \end{tikzpicture} & 
    \begin{tikzpicture}[scale = 0.9]
    \node[thick, draw=blue, ellipse,
    align=center,blue,minimum height=24pt, minimum width = 12pt] at  (0.75,0.5)(0){\footnotesize $\frac{1}{2}$};
    \node[thick, draw=orange, ellipse,
    align=center,minimum height=15pt, minimum width = 10pt] at  (0.75,-0.5)(1){};
    \node[thick, draw = red, fill=red, fill opacity =0.5, text opacity = 1, ellipse,
    align=center,minimum height=24pt, minimum width = 12pt,text = black] at  (1.75,0)(2){$2$};
\node[thick, draw = red, fill=red, fill opacity =0.5, text opacity = 1, ellipse,
    align=center,minimum height=24pt, minimum width = 12pt,text = black] at  (3,0)(3){$1$};
    
    \draw [->] (0) to [out=40,in=180] node[midway,fill = black, circle, inner sep = 0, minimum size =5pt]{}(4.5, 0.4);
    \draw [->] (0) to [out=0,in=150]node[midway,fill = black, circle, inner sep = 0, minimum size =5pt]{}  (2);
    \draw [->] (1) to [out=0,in=-150] (2);
    \draw [->] (1) to [out=-40,in=-140] (3);
    \draw [->] (2) to (3);
    \draw [->] (3) to [out=30,in=180]node[midway,fill = black, circle, inner sep = 0, minimum size =5pt]{}(4.5,0.15);
    \draw [->] (3) to[out=-30,in=180]node[midway,fill = black, circle, inner sep = 0, minimum size =5pt]{} (4.5,-0.15);

    \node[above] at (2.325,0){$2$};
    \node [below] at (4,-0.15){$2$};
    \node[below] at (3.75,-0.5){$(c)$};
    \end{tikzpicture} &
    \begin{tikzpicture}[scale = 0.9]
\usetikzlibrary{shapes}
    \node[thick, draw=blue, ellipse,
    align=center,blue,minimum height=24pt, minimum width = 12pt] at  (1.75,-1)(1){\footnotesize $\frac{1}{2}$};
    \node[thick, draw=blue, ellipse,
    align=center,blue,minimum height=24pt, minimum width = 12pt] at  (1.75,1)(0){\footnotesize $\frac{1}{2}$};
    \node[thick, draw=orange, ellipse,
    align=center,minimum height=15pt, minimum width = 10pt] at  (1.75,0)(2){};\node[thick, draw = red, fill=red, fill opacity =0.5, text opacity = 1, ellipse,
    align=center,minimum height=24pt, minimum width = 12pt,text = black]  at  (3,0)(3){$1$};\
    \draw[->]  (0) to [out =-20, in=130]node[midway,fill = black, circle, inner sep = 0, minimum size =5pt]{} (3);
    \draw[->]  (1) to [out =20, in=-130]node[midway,fill = black, circle, inner sep = 0, minimum size =5pt]{}  (3);
    \draw[->]  (2) to [out =30, in=160] (3);
    \draw[->]  (2) to [out =-30, in=-160] (3);
    \draw [->] (3) to [out=25,in=180]node[midway,fill = black, circle, inner sep = 0, minimum size =5pt]{} (4.5,0.15);
    \draw [->] (3) to[out=-25,in=180] node[midway,fill = black, circle, inner sep = 0, minimum size =5pt]{} (4.5,-0.15);
    \node [above] at (3.75,0.2){$3$};
    \node [below] at (3.75,-0.5){$(b)'$};
    \end{tikzpicture} & 
    \begin{tikzpicture}[scale = 0.9]
    \node[thick, draw=blue, ellipse,
    align=center,blue,minimum height=24pt, minimum width = 12pt] at  (0.75,0.5)(0){\footnotesize $\frac{1}{2}$};
    \node[thick, draw=orange, ellipse,
    align=center,minimum height=15pt, minimum width = 10pt] at  (0.75,-0.5)(1){};\node[thick, draw = red, fill=red, fill opacity =0.5, text opacity = 1, ellipse,
    align=center,minimum height=24pt, minimum width = 12pt,text = black]at  (1.75,0)(2){$2$};
    \node[thick, draw = red, fill=red, fill opacity =0.5, text opacity = 1, ellipse,
    align=center,minimum height=24pt, minimum width = 12pt,text = black] at  (3,0)(3){$1$};
    \draw [->] (0) to [out=40,in=180] node[midway,fill = black, circle, inner sep = 0, minimum size =5pt]{}(4.5, 0.4);
    \draw [->] (0) to [out=0,in=150]node[midway,fill = black, circle, inner sep = 0, minimum size =5pt]{}  (2);
    \draw [->] (1) to [out=0,in=-150] (2);
    \draw [->] (1) to [out=-40,in=-140] (3);
    \draw [->] (2) to [out=30,in=150]node[midway,fill = black, circle, inner sep = 0, minimum size =5pt]{}  (3);
    \draw [->] (2) to [out=-30,in=-150]  (3);
    \draw [->] (3) to [out=30,in=180]node[midway,fill = black, circle, inner sep = 0, minimum size =5pt]{}(4.5,0.15);
    \draw [->] (3) to[out=-30,in=180]node[midway,fill = black, circle, inner sep = 0, minimum size =5pt]{} (4.5,-0.15);
    \draw [->] (3) --node[midway,fill = black, circle, inner sep = 0, minimum size =5pt]{}(4.5,0);
    \node[below] at (3.75,-0.5){$\overline{(c)}$};
    \end{tikzpicture}
\end{tabular}
\end{center}
    \caption{Markings on the floor diagrams $(b), (c)$ and $(b)'$ from Figure \ref{fig-floor-diagrams-expl}. $(b), (c)$ and $\overline{(c)}$ are floor diagrams on $\TT M_0$ and $(b)'$ is on $\TT M_1$.}
    \label{fig:markings on floor diagrams}
\end{figure}
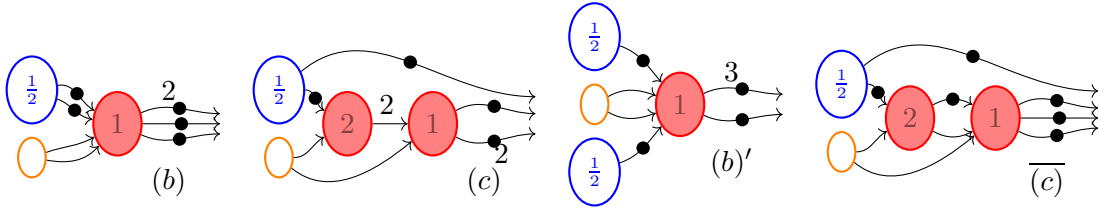

\begin{expl}
We give some examples for markings on floor diagrams in Figure \ref{fig:markings on floor diagrams}.
\end{expl}

\subsection{Floor diagrams from tropical curves}

 We now define a floor diagram for any tropical curve, and show that for a \textit{stretched} choice of constraints, the point conditions induce a marking of the diagram satisfying the condition from Definition \ref{definition marking diagram}, analogous to the cylinder case, \cite{blomme2021floor}. Let $l$ denote  the length of the underlying elliptic curve $\E$.
\begin{prop}
Let $h:\Gamma\to\TT M_\delta$ be a parametrized tropical curve in the class $aE+bF$. Then the slope of its edges can only take a finite number of values, and the lengths of the non-vertical edges are bounded by a constant that only depends on $a,b$ and $l$.
\end{prop}

\begin{proof}
The proof is identical to the proof from \cite{blomme2021floor}, to which we refer the reader.
\end{proof}
\
We now construct a floor diagram from a tropical curve. Let $h:\Gamma\to\TT M_\delta$ be a tropical curve in the class $aE+bF$ with tangency profile $\mu$ and genus $g$. We match the definitions of floor diagrams by calling an edge with vertical slope an \textit{elevator}, and a connected component of the complement of elevators a \textit{floor}. Floors containing a disorienting cycle correspond to \textit{ground floors}, while all other floors are \textit{\'etages}. We can now form a floor diagram where vertices are the floors of the curve, two vertices are linked by an edge if both floors are linked by an elevator,
joints are inserted in the middle of elevators that meet the soul of the M\"obius strip,
 the weight of an elevator is its weight as an edge of $\Gamma$,
and the degree of a floor is equal to half its intersection number with a fiber.

Moreover, assuming the curve passes through a collection of points $\P$, indexed by $[\![1;|\mu|+g-1]\!]$, we have a map $[\![1;|\mu|+g-1]\!]\to\Dfk$ that maps a marked point to the part of $\Dfk$ that contains it.
\begin{expl}
In Figure \ref{fig:constructing floor diagram from tropical curve} we show how to construct floor diagrams from tropical curves on both M\"obius strips for the tropical curves $(c)'$ from Figure \ref{figure curves TM0} and Figure \ref{figure curves TM1} respectively. The resulting floor diagrams are the diagrams $(b)$ and $(b)'$ in Figure \ref{fig-floor-diagrams-expl}. We note that the class, genus and tangency profile of the tropical curves computed in Examples \ref{ex degrees on TM0} and \ref{ex degrees on TM1} coincide with the ones of the corresponding floor diagrams computed in Example \ref{ex genus class tangency floor diagrams}.
\end{expl}
\begin{figure}
\begin{center}
\begin{tabular}{c c}
\begin{tikzpicture}[line cap=round,line join=round,>=triangle 45,x=0.3cm,y=0.3cm,scale = 0.9]
\clip(-1,-1) rectangle (11,12);
\draw [line width=1pt] (0,0)-- (0,15) node[midway,sloped] {$>>$};
\draw [line width=1pt] (10,0)-- (10,15) node[midway,sloped] {$>>$};
\draw [line width=1pt] (0,0) node {$\bullet$} -- (5,0) node[midway,sloped] {$>$} node {$\bullet$} -- (10,0) node[midway,sloped] {$>$} node {$\bullet$};

\draw [blue, line width=2pt] (1,0)--++(1,1)--++(3,0)--++(1,-1);
\draw [red, line width=2pt] (0,4)--++(2,0)--++(1,1)--++(1,-1)--++(1,0)--++(1,1)--++(2,0)--++(1,-1)--++(1,0);
\fill[yellow] (1,4) circle (3pt);
\draw (1,4)  circle (3pt);
\draw [line width=2pt] (2,1)--(2,4);
\fill[yellow] (2,1.75) circle (3pt);
\draw (2,1.75) circle (3pt);
\draw [orange, line width=2pt] (4,0)--(4,4);
\draw [line width=2pt] (5,1)--(5,4);
\fill[yellow] (5,2.75) circle (3pt);
\draw (5,2.75) circle (3pt);
\draw [orange, line width=2pt] (9,0)--(9,4);
\draw [line width=2pt] (3,5)--(3,12) node[midway,right] {$2$};
\fill[yellow] (3,6.25) circle (3pt);
\draw (3,6.25)  circle (3pt);
\draw [line width=2pt] (6,5)--(6,12);
\fill[yellow] (6,9.25) circle (3pt);
\draw (6,9.25)  circle (3pt);
\draw [line width=2pt] (8,5)--(8,12);
\fill[yellow] (8,11) circle (3pt);
\draw (8,11)  circle (3pt);

\draw [dashed,line width=2pt,opacity=0.25] (2.5,0)--(2.5,12);
\draw [dashed,line width=2pt,opacity =0.25] (7.5,0)--(7.5,12);
\end{tikzpicture} &
 \begin{tikzpicture}[line cap=round,line join=round,>=triangle 45,x=0.3cm,y=0.3cm,scale = 0.9]
\clip(-1,-1) rectangle (11,16);
\draw [line width=1pt] (0,0)-- (0,15) node[midway,sloped] {$>>$};
\draw [line width=1pt] (10,5)-- (10,20) node[midway,sloped] {$>>$};
\draw [line width=1pt] (0,0) node {$\bullet$} -- (5,0) node[midway,sloped] {$>$} node {$\bullet$} -- (10,5) node[midway,sloped] {$>$} node {$\bullet$};

\draw [blue, line width=2pt] (1,0)--++(1,1)--++(4,0);
\draw [blue, line width=2pt] (2,0)--++(2,2)--++(3,0);
\draw [red, line width=2pt] (0,4)--++(2,0)--++(1,1)--++(1,2)--++(1,3)--++(2,0)--++(1,-1)--++(2,0);
\fill[yellow] (1,4) circle (3pt);
\draw (1,4)  circle (3pt);
\draw [line width=2pt] (2,1)--(2,4);
\fill[yellow] (2,1.75)circle (3pt);
\draw (2,1.75)  circle (3pt);
\draw [orange,line width=2pt] (3,0)--(3,5);
\draw [line width=2pt] (4,2)--(4,7);
\fill[yellow] (4,3.25) circle (3pt);
\draw (4,3.25)  circle (3pt);
\draw [orange,line width=2pt] (8,3)--(8,9);
\draw [line width=2pt] (5,10)--(5,15) node[midway,left] {$3$};
\fill[yellow] (5,11.25) circle (3pt);
\draw (5,11.25)  circle (3pt);
\draw [ line width=2pt] (7,10)--(7,15);
\fill[yellow] (7,13.75) circle (3pt);
\draw (7,13.75)  circle (3pt);
\draw [dashed,line width=2pt,opacity=0.25] (2.5,0)--(2.5,15);
\draw [dashed,line width=2pt,opacity =0.25] (7.5,2.5)--(7.5,15);
\end{tikzpicture}\\
    \begin{tikzpicture}[scale = 0.9]
\usetikzlibrary{shapes}
    \node[thick, draw=blue, ellipse,
    align=center,blue,minimum height=24pt, minimum width = 12pt] at  (1.75,0.5)(2){\footnotesize $\frac{1}{2}$};
    \node[thick, draw=orange, ellipse,
    align=center,minimum height=15pt, minimum width = 10pt] at  (1.75,-0.5)(1){};
    \node[thick, draw = red, fill=red, fill opacity =0.5, text opacity = 1, ellipse,
    align=center,minimum height=24pt, minimum width = 12pt,text = black] at  (3,0)(3){$1$};\
    
    \draw[->]  (1) to [out =20, in=210](3);
    \draw[->]  (1) to [out =-10, in=230] (3);
    \draw[->]  (2) to [out =10, in=150] node[midway,fill = black, circle, inner sep = 0, minimum size =5pt]{}  (3);
    \draw[->]  (2) to [out =-20, in=170] node[midway,fill = black, circle, inner sep = 0, minimum size =5pt]{}  (3);
    \draw [->] (3) to [out=25,in=180]node[midway,fill = black, circle, inner sep = 0, minimum size =5pt]{} (4.5,0.15);
    \draw [->] (3) to[out=-25,in=180]node[midway,fill = black, circle, inner sep = 0, minimum size =5pt]{}  (4.5,-0.15);
    \draw [->] (3) --node[midway,fill = black, circle, inner sep = 0, minimum size =5pt]{} (4.5,0);
    \node [above] at (3.75,0.2){$2$};
    \node [below] at (3.75,-0.5){$(b)$};
    \end{tikzpicture}&
    \begin{tikzpicture}[scale = 0.9]
\usetikzlibrary{shapes}
    \node[thick, draw=blue, ellipse,
    align=center,blue,minimum height=24pt, minimum width = 12pt] at  (1.75,-1)(1){\footnotesize $\frac{1}{2}$};
    \node[thick, draw=blue, ellipse,
    align=center,blue,minimum height=24pt, minimum width = 12pt] at  (1.75,1)(0){\footnotesize $\frac{1}{2}$};
    \node[thick, draw=orange, ellipse,
    align=center,minimum height=15pt, minimum width = 10pt] at  (1.75,0)(2){};\node[thick, draw = red, fill=red, fill opacity =0.5, text opacity = 1, ellipse,
    align=center,minimum height=24pt, minimum width = 12pt,text = black]  at  (3,0)(3){$1$};\
    \draw[->]  (0) to [out =-20, in=130]node[midway,fill = black, circle, inner sep = 0, minimum size =5pt]{} (3);
    \draw[->]  (1) to [out =20, in=-130]node[midway,fill = black, circle, inner sep = 0, minimum size =5pt]{}  (3);
    \draw[->]  (2) to [out =30, in=160] (3);
    \draw[->]  (2) to [out =-30, in=-160] (3);
    \draw [->] (3) to [out=25,in=180]node[midway,fill = black, circle, inner sep = 0, minimum size =5pt]{} (4.5,0.15);
    \draw [->] (3) to[out=-25,in=180] node[midway,fill = black, circle, inner sep = 0, minimum size =5pt]{} (4.5,-0.15);
    \node [above] at (3.75,0.2){$3$};
    \node [below] at (3.75,-0.5){$(b)'$};
    \end{tikzpicture}
\end{tabular}
\end{center}
    \caption{
The transformation of (marked) tropical curves into (marked) floor diagrams. In blue, we mark disorienting cycles, which correspond to ground floors, in red we mark \'etages. In black, we mark elevators and in orange we mark joints. On the left, we give an example for $\mathbb{T}M_0$: curve $(c')$ in Figure \ref{figure curves TM0}, whereas on the right, we show it on  $\mathbb{T}M_1$, for curve $(c')$ in Figure \ref{figure curves TM1}.
}
    \label{fig:constructing floor diagram from tropical curve}
\end{figure}
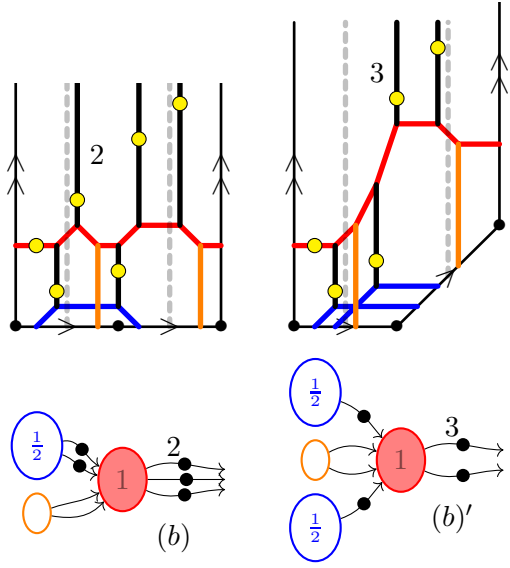
\begin{defi}
A point configuration $\P$ is \textit{stretched} if the difference between the $y$ coordinates of the points is far bigger than $l$, and they are far away from the soul of the M\"obius strip.
\end{defi}

When the point configuration is stretched, the marked floor diagrams  arising as solutions of the enumerative problem satisfy the conditions in Definitions \ref{definition diagram} and \ref{definition marking diagram}.

\begin{prop}
Let $\P$ be a stretched point configuration and $h:\Gamma\to\TT M_\delta$ a tropical curve passing through $\P$. Then, the induced marked diagram $(\Dfk,\mfk)$ satisfies the following:
\begin{enumerate}[label=$(\roman*)$]
\item the floor degrees satisfy $\sum_\F a_\F =a$, and the sum of weights of the ends is equal to $2b$,
\item each  floor satisfies the divergence assumptions in Definition \ref{definition diagram}A and B,
\item each \'etage contains exactly one marked point, and no ground floor contains a marked point,
\item each elevator contains at most one marked point
\item an \'etage consists of a unique cycle with adjacent elevators, and a ground floor consists of a unique disorienting cycle with adjacent elevators,
\item each connected component of the complement of marked elevators is of one of the following types:
	\begin{itemize}[label=$\ast$,noitemsep]
	\item It contains a unique ground floor and no cycles or free ends;
	\item It contains a unique free end and no cycles or ground floors;
	\item It has a unique cycle containing an odd number of joints and no free ends or ground floors.
	\end{itemize}
\end{enumerate}
In particular, $\Dfk$ is a floor diagram and $\P$ induces a marking $\mfk$ of $\Dfk$.
\end{prop}

\begin{proof}
\begin{enumerate}[label=$(\roman*)$]
\item This follows from the definition of the degree of the floors and class of the curve.
\item The statement for \'etages follows using the balancing condition in Definition \ref{definition diagram}B. For ground floors, we pass to the two-to-one cover of the M\"obius strip, which is $\TT\E\times\RR$ for $\TT M_0$ and the total space of the $2$-torsion line bundle on $\TT\widetilde{E}$ for $\TT M_1$. As the ground floor contains a disorienting cycle, the preimage of a neighbourhood of the ground floor by the two-to-one cover is connected. Each  elevator adjacent to the ground floor yields a pair of elevators in the cover, whose coordinates on $\TT\widetilde{E}=\RR/2l\ZZ$ differ by $l$ and which are in opposite direction. Let $x_e\in\RR/2l\ZZ$ denote the coordinate of an adjacent elevator $e$. Using the tropical Menelaus relation for the cylinder $\TT\widetilde{E}\times\RR$  proven in \cite{blomme2021floor}, we get that
$$\sum_e w_e x_e -w_e(x_e+l)\equiv l\cdot\delta\cdot 2 a_\G \text{ mod }2l.$$
This yields the desired relation for the divergence of grounds floors.

\item An \'etage has to contain at least one marked point, otherwise it is possible to translate it vertically, resulting in a $1$-parameter family of solutions. It cannot contain more than one point, as the point conditions are stretched whereas slope and length of non-vertical edges are bounded. Further, as a disorienting cycle intersects the soul of the M\"obius strip and the points are chosen very far from it,  a ground floor cannot contain a marked point.

\item As the point configuration is generic, no points have the same projection onto $E$, and an elevator cannot contain more than one marked point.

\item Any cycle contained in an \'etage is orienting. If an \'etage contained two cycles, one of them would be without marked points, contradicting that $\Gamma\backslash h^{-1}(\P)$ contains no orienting cycle (see Proposition \ref{proposition shape of solution}). If a ground floor contained two cycles, it would again contain an orienting cycle. By the balancing condition, a disconnecting edge in the \'etage would have vertical slope, hence there are no disconnecting edges, leading to the statement.

\item The last statement is the direct translation of Proposition \ref{proposition shape of solution}, following from the observation that disorienting cycles lie either in a ground floor, or  can use elevators, as long as they intersect the soul of the M\"obius strip an odd number of times in total.
\end{enumerate}
\end{proof}

\subsection{Floor diagram multiplicities}

Now, we recover the marked tropical curves that are solutions of the enumerative problem from the floor diagrams. Definition \ref{definition diagram multiplicity} below gives the (refined) multiplicity of a floor diagram so that it matches the sum of (refined) multiplicities of the tropical curves it encodes (see Proposition \ref{proposition diagram multiplicity}). We set
\begin{equation}\label{eq sigma one}
\widetilde{\sigma_1}(a)=\sum_{\substack{ k|a \\ k\text{ odd}}}\frac{a}{k}.
\end{equation}

\begin{defi}\label{definition diagram multiplicity}
Let $(\Dfk,\mfk)$ be a marked floor diagram. If $\F$ is an \'etage and $\G$ is a ground floor, we set
\begin{equation}
\begin{array}{>{\displaystyle}l>{\displaystyle}l}
m(\F)=a_\F^{\val\F-1}\sigma_1(a_\F)\prod_{e\ni\F}w_e, & m^q(\F)=\sum_{k|a_f} k^{\val\F-1}\prod_{e\ni\F}\quant{\frac{w_e a_\F}{k}} \\
m(\G)=2\cdot (2a_\G)^{\val\G-1}\widetilde{\sigma_1}(2a_\G)\prod_{e\ni\G}w_e  & m^q(\G)= 2\sum_{\substack{k|2a_\G \\ k \text{ odd}}} k^{\val\G-1}\prod_{e\ni\G}\quant{\frac{w_e\cdot 2 a_\G}{k}} 
\end{array}    
\end{equation}
Further, we write $N(\Dfk)$ for the number of cycles in the complement of marked elevators, and and $\mathrm{Aut}\Dfk$ for the group of automorphisms of the diagram. We define the (refined) multiplicity of a marked floor diagram to be
$$\begin{array}{>{\displaystyle}l}
m(\Dfk,\mfk)=\frac{2^{N(\Dfk)}}{|\mathrm{Aut}\Dfk|}\prod_\F m(\F)\prod_\G m(\G)\prod_{E_\mathrm{um}}w_e, \\
m^q(\Dfk,\mfk)=\frac{2^{N(\Dfk)}}{|\mathrm{Aut}\Dfk|}\prod_\F m^q(\F)\prod_\G m^q(\G)\prod_{E_\mathrm{um}} w_e. 
\end{array}$$
\end{defi}

\begin{rem}
The multiplicity of a tropical curve is a product over its vertices. In the classical floor diagrams \cite{brugalle2007enumeration,brugalle2008floor} as well as their refined version \cite{block2016refined}, the weight of the edges is enough to determine the multiplicity, which is a product over the edges of the diagram. Here, as in \cite{blomme2021floor}, we have to further account for the floors. In the classical case, it is possible to factor out the weight of the edges from the floors, but not in the refined case.
\end{rem}

\begin{prop}\label{proposition diagram multiplicity}
The (refined) multiplicity of a marked floor diagram corresponds to the (refined) count of tropical curves that it encodes, counted with (refined) multiplicity.
\end{prop}

\begin{proof}
We recover tropical curves from floor diagrams proceeding as follows. First, we determine the shape of the floors.

\begin{itemize}[label=$\ast$,noitemsep]
\item Given an \'etage $\F$ of degree $a_\F$, the situation is handled as for floors on cylinders \cite{blomme2021floor}. An \'etage consists of a unique cycle realizing an even homology class  $2k_\F\in H_1(\TT M_\delta)\simeq\ZZ$ in the M\"obius strip. The horizontal coordinate $v_\F$ of the slope of the edges in the cycle is well-defined. In particular, an elevator of weight $w_e$ that meets the cycle does so at a vertex with multiplicity $w_ev_\F$. Moreover, intersecting with a fiber yields $2k_\F$ intersection points each of multiplicity $v_\F$. As the intersection index is by definition $2a_\F$, we have $a_\F=k_\F v_\F$, hence $k_\F|a_\F$.

Given $k_\F|a_\F$, we can unfold the \'etage so that the cycles goes  around the M\"obius strip only twice. This induces a floor if and only if the position of the adjacent elevators satisfy the \textit{Menelaus relation} \cite{blomme2021floor} in $\RR/2lk_\F\ZZ$: if $x_e$ is the position of the elevator $e$, we require
	$$\sum_{e\ni\F}\pm w_e x_e\equiv \delta v_\F l \in \RR/2lk_\F\ZZ.$$

\item Given a ground floor $\G$ of degree $a_\G$, there also is a unique cycle that needs to be disorienting, i.e. it realizes an odd homology class in the M\"obius strip. Let $k_\G\in H_1(\TT M_\delta)\simeq\ZZ$ be this odd homology class. The horizontal coordinate $v_\G$ of the slope of the edges in the cycle is also well-defined. Intersecting with a fiber yields $k_\G$ intersection points each of multiplicity $v_\G$.  By the definition of the degree, we obtain $k_\G v_\G=2a_\G$. Therefore, $k_\G$ is an odd divisor of $2a_\G$.

Thus, we can unfold the ground floor by the cover of degree $k_\G$, so that it goes around the M\"obius strip only once. Hence, we can assume that $k_\G=1$ up to choosing a lift of the elevators by this cover. Then, we can take the preimage by the two-to-one cover of the M\"obius strip by the cylinder. Each adjacent elevator gets two lifts with opposite directions whose horizontal position differs by $l$. The Menelaus condition on the cylinder is
$$\sum_{e\ni\G}w_e (x_e - (x_e-l))=\delta l \in\RR/2l\ZZ.$$
This relation is automatically satisfied by the balancing condition, so that we have a unique curve up to translation. If require the curve to be invariant by the deck transformation of the cover, it is fully unique. In the end, we can draw a unique curve for each choice of lift of the elevators.
\end{itemize}
Now, we recover the positions of the elevators for each floor and fixed $k_\F$ as the lattice index of the following map. For each elevator $e$, let $e_+$ and $e_-$ be its extremities, which are floors, joints or points on the boundary of the strip. If $e_-$ is a joint, we set $k_{e_-}=1$ and if $e_+$ is a boundary point, we set $k_{e_+}=1$. We consider the space of positions of the elevators in the unfolded version$$\prod_{e} \RR/2lk_{e_+}\ZZ \times \RR/2lk_{e_-}\ZZ.$$
An element $(x_{e_+},x_{e_-})_e$ can only correspond to a tropical curve if the following are satisfied:
	\begin{itemize}[label=$\bullet$,noitemsep]
	\item For each edge $e$, $x_{e_+}\equiv x_{e_-}$ in $\RR/2l\ZZ$, so both extremities can linked by an elevator.
	\item For each joint, $x_{e_+}$ and $x_{e_-}$ differ by $l\in\RR/2l\ZZ$.
	\item For each \'etage and two adjacent elevators $e$ and $e'$, $(x_{e_+},x_{e_-})_e$ and $(x_{e'_+},x_{e'_-})_{e'}$ satisfy the unfolded Menelaus relation in $\RR/2l k_\F\ZZ$.
	\item Marked points and fixed ends fix the position of elevators.
	\end{itemize}
	Finally, we have the map of real tori
	$$\begin{array}{r>{\displaystyle}cc>{\displaystyle}l}
	\Phi: & \prod_{e} \RR/2lk_{e_+}\ZZ \times \RR/2lk_{e_-}\ZZ
	 & \longrightarrow & \prod_e \RR/2l\ZZ \times \prod_\F \RR/2l k_\F\ZZ\times\prod_\J \RR/2l\ZZ \times \prod_{e\text{ marked}}\RR/2l\ZZ. \\
	 & (x_{e_+},x_{e_-}) & \longmapsto & \left( (x_{e_+}-x_{e_-}),(\sum \pm w_e x_{e_\pm}),(x_{e(\J)_-} -x_{e'(\J)_-}),(x_{e(i)_+}) \right) \\
	\end{array}$$
	This is a group homomorphism between real tori of the same dimension. We now the number of preimages of an element $\left( (0),(\delta v_\F l),(l),(x_i)\right)$ by computing the lattice index of the map between the first homology groups of the tori. The lattice index of $\Phi_*$ can be computed similarly as the computation from the multiplicity in Proposition \ref{prop-computation-mult}. We prune the diagram using Lalace expansion formula for determinants. In the end, we get
$$2^{N(\Dfk)}\prod_\F k_\F^{\val\F-1}\prod_{e\text{ unmarked}}w_e,$$
    where $N(\Dfk)$ denotes the number of cycles in the complement of marked points in $\Dfk$. To conclude, we multiply the lattice index of $\Phi_*$ by the multiplicity of a curve encoded by the diagram, and make the sum over the possible divisors $k_\F|a_\F$ and $k_\G|2a_\G$, yielding the result.

\end{proof}

\begin{expl}
    We conclude by applying the floor diagram algorithm in the genus one case. In genus one, every diagram has a unique floor, either a ground floor or an \'etage.
    
    In the case of a unique ground floor, every elevator is adjacent to it. The contribution is $(2a)^{2b}\widetilde{\sigma}_1(2a)$ respectively.

    In the case of a unique \'etage, every elevator is still adjacent to the \'etage, but might be passing through a joint. By balancing, there are as many elevators directly adjacent to the \'etage as elevators passing through a joint before going to the \'etage. Thus, $b$ needs to be even. We count the number of markings by considering the lift under the 2-to-1 cover of the half-line to get a floor diagram in a cylinder, also with a unique floor. The (even) number of ends $2b$ coincides with the number of marked points. There are $2^{2b-1}$ lifts of the points in the cylinder up to the deck transformation. In the lift there are precisely $2$ markings, depending on where the free end lie, as given in \cite{blomme2021floor}. The multiplicity is thus $2^{2b-1}\cdot 2\cdot a^{2b}\sigma_1(a)$.

    Summing over all contributions, we get
    $N^\delta_{1,aE+bF}=(2a)^{2b}\left(\widetilde{\sigma}_1(2a)+[a,b\in\ZZ]\sigma_1(a)\right),$
    where $[a,b\in\ZZ]$ is $1$ if $a$ and $b$ are integers, and $0$ else.
\end{expl}

\section{Regularity of invariants}

\subsection{Quasi-polynomiality of relative invariants}
\label{sec polynomiality}

	In this section we study the regularity of the relative invariants, fixing the number of intersection points but varying their tangency orders. This was previously done in \cite{ardila2017double} for the relative invariants of Hirzebruch surfaces and in \cite{blomme2021floor} for the case of line bundles over an elliptic curve. Other results on polynomiality have recently been obtained in \cite{corey2022counting}.
	
	\subsubsection{General statement.} We study the function
	$$\Phi_a^\delta:(\mu_1,\dots,\mu_n,\nu_1,\dots,\nu_m)\longmapsto N^\delta_{g,aE+bF}(\mu_1,\dots,\mu_n,\nu_1,\dots,\nu_m),$$
	defined for $\delta=0,1$, and $a\in\frac{1}{2}\NN$. To get a non-zero result, we must have $\sum\mu_j+\sum\nu_j=2b$. Hence, $b$ is chosen accordingly. In particular,
	$$\sum\mu_j+\sum\nu_j\equiv 2\delta a \text{ mod }2.$$
Therefore, we consider the function $\Phi_a^\delta$ to be defined only on tuples of integers satisfying the above conditions. The study of regularity relies on the existence of floor diagrams, counting each with a polynomial contribution. Different to \cite{ardila2017double} and \cite{blomme2021floor}, diagrams with ground floors and joints often only have \textit{quasi-polynomial} (but not polynomial) contributions.

\begin{theo}\label{theorem-quasi-polynomiality}
There exist piecewise quasi-polynomial functions $P^\delta_a$ in $n+m$ variables such that $\Phi_a^\delta(\mu,\nu)=P_a^\delta(\mu,\nu)$.
\end{theo}

\begin{proof}
The proof relies on \cite[Theorem 1]{sturmfels1995vector}. We proceed as in \cite{ardila2017double}. The relative invariant can be written as a sum over the floor diagrams. As the curves are of fixed genus and have a fixed number of ends, up to the weighting of the elevators, there is a finite number of floor diagrams.

Let $\Dfk$ be a marked floor diagram. We label the ends of $\Dfk$ by the coordinates $(\mu,\nu)$. As in \cite{ardila2017double}, we  account for the symmetry in the partition $\nu$ by averaging over its automorphism group. Our goal is to find the possible weightings of the internal edges of $\Dfk$. To incorporate the parity conditions on the edges adjacent to ground floors, we modify $\Dfk$ to get a graph $G_\Dfk$ by adding a vertex $v_e$ adjacent to each end $e$ and an unbounded edge $e_\G$ adjacent to each ground floor $\G$.

Let $E$ be the set of edges of $G_\Dfk$ and $V$ its set of vertices. The graph $G_\Dfk$ inherits an orientation from $\Dfk$. Now, a \emph{weighting} of $\Dfk$ is a vector in $\NN^E$ satisfying the balancing condition at each floor and the equality of weights of edges adjacent to the same joint. Let $\varepsilon_\G$ be $2\delta a_\G$ modulo $2$. Assume  that the balancing condition for ground floors is satisfied on $G$. We obtain a compatible weight $w_{e_\G}$ for the new edge $e_\G$ by solving
	$$2w_{e_\G}+\varepsilon_\G=\sum_{e\ni\G \text{ in }\Dfk} w_e.$$
    Let $A$ be the adjacency matrix of the oriented graph $G_\Dfk$. The coefficient of $A$ for $e_\G\ni\G$ is $2$ by balancing. Let $\mathbf{d}\in\ZZ^V$ be the integer vector whose coordinates are given by
	$$d_v=\left\{ \begin{array}{l}
	0 \text{ if }v\text{ is an \'etage}, \\
	\varepsilon_\G  \text{ if }v\text{ is a ground floor}, \\
	\mu_i \text{ or }\nu_i \text{ if }v=v_e\text{ for some end }e.
	\end{array}	  \right.$$
	The multiplicity of a floor diagram is a monomial in the coordinates of the weight vector $\mathbf{w}$. Thus, we count the weight vectors $\mathbf{w}\geqslant 0$ satisfying $A\mathbf{w}=\mathbf{d}$ with corresponding multiplicity.
    
    If there are no ground floors and every cycle passes through an even number of joints, we can lift the floor diagram to a floor diagram on the cylinder. This corresponds to lifting the tropical curves encoded by the diagram to tropical curves in the two-to-one cylinder cover of the M\"obius strip. On the cylinder, the sum of top weights is equal to the sum of bottom weights, thus we can find a weighting only if this balancing condition is satisfied, and polynomiality is due to \cite{blomme2021floor}.
	
    If there is at least one ground floor or a cycle passing through an odd number of joints, the map is surjective over $\QQ$, i.e. the image of $A$ is of full-dimension.
    
	By \cite{sturmfels1995vector}, the function is a piecewise quasi-polynomial on the chamber complex of the matrix $A$. On each chamber, the function is polynomial on vectors with fixed residue modulo all non-zero principal minors of $A$. Lemma \ref{lem-computation-minor-adjacency-matrix} gives a more complete description of the minors.
\end{proof}

\begin{rem}
 The statement in \cite{sturmfels1995vector} concerns the count without polynomial multiplicity. The extended statement can be obtained by induction on the degree by considering the graph where some of the edges have been doubled. For details, see the proof of Theorem 4.2 in \cite{ardila2017double}.
\end{rem}

\begin{center}
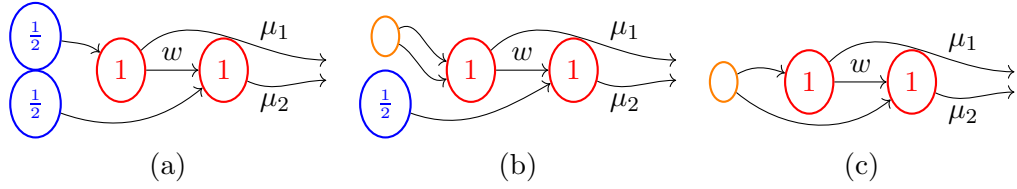
\begin{figure}
\begin{center}
\begin{tabular}{c c c}
    \begin{tikzpicture}[scale = 0.9]
\usetikzlibrary{shapes}
    \node[thick, draw=blue, ellipse,
    align=center,blue,minimum height=24pt, minimum width = 12pt] at  (0.75,-0.5)(1){\footnotesize $\frac{1}{2}$};
    \node[thick, draw=blue, ellipse,
    align=center,blue,minimum height=24pt, minimum width = 12pt] at  (0.75,0.5)(0){\footnotesize $\frac{1}{2}$};
    \node[thick, draw=red, ellipse,
    align=center,minimum height=24pt, minimum width = 12pt,red] at  (2,0)(2){$1$};
    \node[thick, draw=red, ellipse,
    align=center,minimum height=24pt, minimum width = 12pt,red] at  (3.5,0)(3){$1$};
    \draw[->]  (0) to [out =-10, in=140] (2);
    \draw[->]  (1) to [out =-20, in=-140] (3);
    \draw[->]  (2) to (3);
    \draw [->] (2) to [out=45,in=180](5,0.15);
    \draw [->] (3) to[out=-25,in=180] (5,-0.15);
    \node [above] at (4.25,0.3){$\mu_1$};
    \node [above] at (2.75,0){$w$};
    \node [below] at (4.25,-0.2){$\mu_2$};
    \end{tikzpicture} &     \begin{tikzpicture}[scale = 0.9]
\usetikzlibrary{shapes}
    \node[thick, draw=blue, ellipse,
    align=center,blue,minimum height=24pt, minimum width = 12pt] at  (0.75,-0.5)(0){\footnotesize $\frac{1}{2}$};
    \node[thick, draw=red, ellipse,
    align=center,minimum height=24pt, minimum width = 12pt,red] at  (2,0)(2){$1$};
    \node[thick, draw=red, ellipse,
    align=center,minimum height=24pt, minimum width = 12pt,red] at  (3.5,0)(3){$1$};
        \node[thick, draw=orange, ellipse,
    align=center,minimum height=15pt, minimum width = 10pt] at  (0.75,0.5)(1){};

    \draw[->]  (1) to [out =30, in=160] (2);
    \draw[->]  (1) to [out =-30, in=-160] (2);

    \draw[->]  (0) to [out =-20, in=-150] (3);
    \draw[->]  (2) to (3);
    \draw [->] (2) to [out=45,in=180](5,0.15);
    \draw [->] (3) to[out=-25,in=180] (5,-0.15);
    \node [above] at (4.25,0.3){$\mu_1$};
    \node [above] at (2.75,0){$w$};
    \node [below] at (4.25,-0.2){$\mu_2$};
    \end{tikzpicture} &     \begin{tikzpicture}[scale = 0.9]
\usetikzlibrary{shapes}
    \node[thick, draw=red, ellipse,
    align=center,minimum height=24pt, minimum width = 12pt,red] at  (2,0)(2){$1$};
    \node[thick, draw=red, ellipse,
    align=center,minimum height=24pt, minimum width = 12pt,red] at  (3.5,0)(3){$1$};
        \node[thick, draw=orange, ellipse,
    align=center,minimum height=15pt, minimum width = 10pt] at  (0.75,0)(1){};

    \draw[->]  (1) to [out =30, in=160] (2);
    \draw[->]  (1) to [out =-40, in=-140] (3);
    \draw[->]  (2) to (3);
    \draw [->] (2) to [out=45,in=180](5,0.15);
    \draw [->] (3) to[out=-25,in=180] (5,-0.15);
    \node [above] at (4.25,0.3){$\mu_1$};
    \node [above] at (2.75,0){$w$};
    \node [below] at (4.25,-0.2){$\mu_2$};
    \end{tikzpicture}\\
    (a) & (b) & (c) \\    
    \end{tabular}
    \end{center}
    \caption{Floor diagrams with a non-polynomial contribution.}
    \label{fig-floor-diag-nonpoly-mult}
\end{figure}
\end{center}

\begin{expl}\label{expl-diag-with-non-poly-mult}
In the following, we give some examples of floor diagrams with a non-polynomial contribution. Their graphs are depicted in Figure \ref{fig-floor-diag-nonpoly-mult}.
	\begin{enumerate}[label=(\alph*),noitemsep]
	\item[\ref{fig-floor-diag-nonpoly-mult}(a)]  Here, for given $\mu_1$ and $\mu_2$, complete the diagram by choosing the weights of the remaining edges. The choice of  the interior edge $w$ completely determines the weight of the edges adjacent to the ground floors. However, we have a parity condition on these weights: even (resp. odd) if the diagram encodes curves in $\TT M_0$ (resp. $\TT M_1$). Thus, $w$ needs to have the same (resp. opposite) parity as $\mu_1$. Recall that the parity of $\mu_1+\mu_2$ is fixed.
	
	\item[\ref{fig-floor-diag-nonpoly-mult}(b)] For the second diagram, we have a similar parity obstruction: since a joint is adjacent to the first \'etage, we need to have $w\equiv \mu_1 \text{ mod }2$.
	
	\item[\ref{fig-floor-diag-nonpoly-mult}(c)] For the last diagram, we can uniquely solve for the weights of the edges.\qedhere
	\end{enumerate}	 
\end{expl}
\subsubsection{Description of the minors of the adjacency matrix.} Given a marked floor diagram $\Dfk$ and an adjacency matrix $A$ of $G_\Dfk$, let $A_\Delta$ be a principal minor of $A$ corresponding to a subset $\Delta\subset E$ of edges and let $G_\Delta$ be the subgraph of $G_\Dfk$ containing all the vertices and the edges of $\Delta$. The quotient $\ZZ^V/\gen{A_\Delta}$ is computed by the following lemma.

\begin{lem}\label{lem-computation-minor-adjacency-matrix}
In the above notation,
	\begin{enumerate}[label=$(\roman*)$]
	\item If the determinant of the minor is non-zero, then $G_\Delta$ is a disjoint union of connected subgraphs which are either trees containing a unique $e_\G$, or which contain a unique cycle and no edge $e_\G$.
	\item The group $\ZZ^V/\gen{A_\Delta}$ is a sum over the connected components of the subgraph $G_\Delta$, where the group associated to a component is $\ZZ/2\ZZ$ if the component is a tree containing some $e_\G$, or a cycle passing through an odd number of joints, and $0$ else (i.e. a cycle passing through an even number of joints).
	\end{enumerate}
\end{lem}

\begin{proof}
Choosing a principal minor amounts to choosing a subset $\Delta\subset E$ of size $|V|$. If the principal minor is non-zero, $\Delta$ has a bijection ot $V$ that assigns each vertex one to of its adjacent edges in $\Delta$. The expansion of the determinant is now a signed weighted sum over these bijections. As the number of vertices and edges in  $G_\Delta$ is the same, $G_\Delta$ has Euler characteristic $0$, and the same holds true for each of its connected components. Thus, for each connected component there are two possibilities: Either it is a tree rooted at some edge corresponding to a ground floor or it contains a unique cycle.
 
	This proves $(i)$. To show $(ii)$, note that the minor $A_\Delta$ splits as a block-diagonal matrix corresponding to its connected components. We claim that the determinant of the block corresponding to a connected component is either $1$ or $2$, so that the quotient is a sum of copies of $\ZZ/2\ZZ$. We can compute each determinant by pruning the branches of the components, which amounts to Laplace expansion with respect to the row corresponding to the vertex we prune. The coefficient of the vertex being $\pm 1$, we end up with one of these two cases:
	\begin{itemize}[label=$\ast$]
	\item For a component without cycle and rooted at an end $e_\G$, the determinant is $2$ since the coefficient of the latter in $A$ is $2$.
	\item If the component has a unique cycle, pruning the branches, we are left with the cycle. When expanding the determinant, we have exactly two terms according to the choice of vertex-edge assignment. The value of this determinant is $0$ if the cycle passes through an even number of joints and $2$ if this number is odd.
	\end{itemize}
\end{proof}

\subsubsection{Piecewise polynomiality in some situations.}  We now use this fact to prove polynomiality in several special cases. First, we consider curves that have a unique tangency point with maximal order on the boundary. Afterwards, we show the analogue for curves of small genus.

\begin{coro}\label{coro-one-end}
The relative invariants having a unique intersection point with the boundary are piecewise polynomial.
\end{coro}

\begin{proof}
There is a unique variable $\mu=2b\in\NN$. By Theorem \ref{theorem-quasi-polynomiality}, for each floor diagram, the contribution is piecewise quasi-polynomial, and the quasi-polynomiality is determined by the residue of the divergence vector $\mathbf{d}$ modulo the principal minors of the adjacency matrix $A$. To conclude, we only need to show that these residues are constant. Let us consider a principal minor $A_\Delta$ corresponding to a subgraph $G_\Delta$. According to Lemma \ref{lem-computation-minor-adjacency-matrix}, the cokernel is a sum of copies of $\ZZ/2\ZZ$ corresponding to the components with either a cycle with an odd number of joints, or an end $e_\G$ at a ground floor. Thus, only the residue mod $2$ of the divergence vector $\mathbf{d}$ matter. All of its coordinates are constant except $\mu$, which is either $0$ or $\varepsilon_\G$. As $\mu=2b\equiv 2\delta a$ mod $2$, its residue mod $2$ is fixed. The function is thus a true polynomial.
\end{proof}

\begin{coro}\label{coro-genus12}
The relative invariants of genus $1$ and $2$ are piecewise polynomials.
\end{coro}

\begin{proof}
We proceed similarly to prove that that all the residues are fixed. Assume the diagram is of genus $1$. It has at most one floor; As it needs to have at least one, the floor is unique.

Assume the unique floor is an \'etage. Every end is adjacent to it, possibly through some joint. Note that this diagram only contributes if the balancing condition at the \'etage is satisfied, which imposes a condition on the ends. Then the matrix is unimodular, which, by \cite{ardila2017double}, implies that the contribution is polynomial. 

If the unique floor is a ground floor, every end is (directly) adjacent to it. There is a unique minor and the quotient is $\ZZ/2\ZZ$. However, as we have $\sum\mu_i=2b\equiv 2\delta a$ mod $2$, the residue is constant, implying polynomiality.

In case of genus $2$, the possibilities for the floor diagrams are given in Figure \ref{figure floor diagrams genus 2}:
	\begin{itemize}[label=$\ast$]
	\item[\ref{figure floor diagrams genus 2}(a)] If the diagram has a unique floor, it needs to be an \'etage with a joint whose both extremities are adjacent to it. Further, every end is adjacent to the \'etage, after potentially passing through some joint. As there is a cycle passing through an odd number of joints (i.e. $1$), the determinant is $2$. The residue is given by the sum of the entries, which is fixed by $a$ since $2b\equiv 2\delta a\text{ mod }2$. Thus, we get polynomiality.
	\item[\ref{figure floor diagrams genus 2}(b)] Assume the diagram has two floors, both of which are \'etages. The matrix is unimodular, again imlying polynomiality.
	\item[\ref{figure floor diagrams genus 2}(c)] Assume the diagram has a ground floor and an \'etage.  The ground floor is linked to the \'etage by an unique edge. Thus, we have exactly one principal minor with determinant $2$, and the residue is given by the sum of the entries. Thus, it is also fixed and we get polynomiality.
	\end{itemize}
\end{proof}

\begin{figure}
\begin{center}
\begin{tabular}{c c c}
\begin{tikzpicture}[scale = 0.9]
    \node[thick, draw=orange, ellipse,
    align=center,minimum height=15pt, minimum width = 10pt] at  (0,0)(0){};
    \node[thick, draw=red, ellipse,
    align=center,minimum height=24pt, minimum width = 12pt,red] at  (1.5,0)(1){$1$};
    \draw [->](0) to [out=30,in=150] (1);
    \draw [->] (0) to [out=-30,in=-150] (1);
    \draw [->] (1) to [out=30,in=180](3,0.15);
    \draw [->] (1) to[out=-30,in=180] (3,-0.15);

    \node[thick, draw=orange, ellipse,
    align=center,minimum height=15pt, minimum width = 10pt, densely dotted] at  (0,-0.75)(4){};
    \draw[->, densely dotted] (4) to [out=30,in=-120] (1);
    \draw[->, densely dotted] (4) to [out=0,in=180] (3,-0.45);
    \end{tikzpicture} &
\begin{tikzpicture}[scale = 0.9]
    \node[thick, draw=orange, ellipse,
    align=center,minimum height=15pt, minimum width = 10pt] at  (0,0)(0){};
    \node[thick, draw=red, ellipse,
    align=center,minimum height=24pt, minimum width = 12pt,red] at  (1.5,0)(1){$1$};
    \node[thick, draw=red, ellipse,
    align=center,minimum height=24pt, minimum width = 12pt,red] at  (3,0)(2){$1$};
    \draw [->](0) to [out=30,in=150] (1);
    \draw [->] (0) to [out=-30,in=-150] (2);
    \draw [->] (1) to [out=45,in=180](4.5,0.15);
    \draw [->] (1) to[out=50,in=180] (4.5,0.45);
    \draw [->] (2) to [out=-15,in=180](4.5,-0.15);
    \draw [->] (2) to[out=-30,in=180] (4.5,-0.45);

    \node[thick, draw=orange, ellipse,
    align=center,minimum height=15pt, minimum width = 10pt, densely dotted] at  (0,-0.75)(4){};
    \draw[->, densely dotted] (4) to [out=20,in=-130] (2);
    \draw[->, densely dotted] (4) to [out=-10,in=180] (4.5,-0.75);
    \node[thick, draw=orange, ellipse,
    align=center,minimum height=15pt, minimum width = 10pt, densely dotted] at  (0,0.75)(4){};
    \draw[->, densely dotted] (4) to [out=-20,in=120] (1);
    \draw[->, densely dotted] (4) to [out=10,in=180] (4.5,0.75);
    \end{tikzpicture}   & 
\begin{tikzpicture}[scale = 0.9]
    \node[thick, draw=blue, ellipse,
    align=center,blue,minimum height=24pt, minimum width = 12pt] at  (0,0)(0){\footnotesize $\frac{1}{2}$};
    \node[thick, draw=red, ellipse,
    align=center,minimum height=24pt, minimum width = 12pt,red] at  (1.5,0)(1){$1$};
    \draw [->, dashed](0) to [out=30,in=180] (3,0.45);
    \draw [->] (0) to [out=-30,in=-150] (1);
    \draw [->] (1) to [out=30,in=180](3,0.15);
    \draw [->] (1) to[out=-30,in=180] (3,-0.15);
    
    \node[thick, draw=orange, ellipse,
    align=center,minimum height=15pt, minimum width = 10pt, densely dotted] at  (0,-0.75)(4){};
    \draw[->, densely dotted] (4) to [out=30,in=-120] (1);
    \draw[->, densely dotted] (4) to [out=0,in=180] (3,-0.45);
    \end{tikzpicture}\\
    $(a)$ & $(b)$ & $(c)$ \\
\end{tabular}
\caption{\label{figure floor diagrams genus 2} All floor diagrams of genus 2 and tangency profile $1^{2a}$. An arbitrary degree $2a$ can be reached by attaching ends via joints as indicated by the dotted part of the pictures. Floor diagrams $(a)$ and $(b)$ contribute for both $\TT M_0$ and $\TT M_1$, the diagram in $(c)$ contributes with the dashed end for $\TT M_0$ and without the dashed end for $\TT M_1.$}
\end{center}
\end{figure}
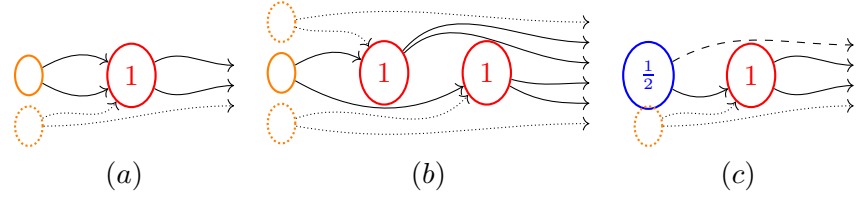

\subsection{Quasi-modularity}

 Quasi-modularity statements for the generating series of invariants of degree $0$ line bundle over an elliptic curve have been proven in \cite{blomme2021floor}. Quasi-modularity is a desirable property since it implies a strong control over the coefficients, bounding their growth polynomially. We consider the generating series of $N^{\delta}_{g,aE+bF}(\mu,\nu)$ in $a$. For $\TT M_0$, given $b$ and partitions $\mu+\nu\vdash 2b$, we set
$$F^0_{g,b}(\mu,\nu)(y) =\sum_{a\in\frac{1}{2}\NN} N^0_{g,aE+bF}(\mu,\nu)y^{2a},$$
where $g\geqslant 1$ and $b\in\NN$. We consider exponents $2a$ of the variable $y$ since $a$ is a half-integer. Similarly, for $\TT M_1$, as $b$ can be an integer or a half-integer, we have two generating series: we set
\begin{align*}
F^1_{g,b}(\mu,\nu)(y) & =\sum_{a\in\NN} N^1_{g,aE+bF}(\mu,\nu)y^{2a} \text{ if }b\in\NN, \\
F^1_{g,b}(\mu,\nu)(y) & =\sum_{a\in\frac{1}{2}+\NN}^\infty N^1_{g,aE+bF}(\mu,\nu)y^{2a} \text{ if }b\notin\NN. 
\end{align*}
In other words, we consider the generating series of relative invariants, fixing the intersection profile with the boundary and varying the intersection number $(aE+bF)\cdot F=2a$, which is the exponent of the series variable.

Before proving the regularity result on the above generating series, we introduce the following auxiliary functions:
	\begin{itemize}[label=$\ast$,noitemsep]
	\item $\displaystyle G_2(y)=\sum_{n=1}^\infty \sigma_1(n)y^n$, the usual Eisenstein series up to an affine transformation,
	\item $\displaystyle H(y)=\sum_{n=1}^\infty \widetilde{\sigma_1}(n)y^n$, the generating series of $\displaystyle \widetilde{\sigma_1}(n)=\sum_{\substack{k|n \\ k\text{ odd}}}\frac{n}{k}$,
	\item $\displaystyle H_0(y)=\sum_{n=1}^\infty \widetilde{\sigma_1}(2n)y^{2n},\ H_1(y)=\sum_{n=0}^\infty \widetilde{\sigma_1}(2n+1)y^{2n+1}$, the odd and even parts of $H(y)$.
	\end{itemize}
	
	\begin{lem}
	The functions $H$, $H_0$ and $H_1$ are quasi-modular forms for some finite index subgroup of $SL_2(\ZZ)$.
	\end{lem}
	
	\begin{proof}
	We start with the function $H$. We have:
	\begin{align*}
	H(y)= & \sum_{n=1}^\infty \left(\sum_{\substack{k| n \\ k\text{ odd}}} \frac{n}{k}\right)y^n 
	=  \sum_{n=1}^\infty \left(\sum_{k|n}\frac{n}{k}\right)y^n -\sum_{n=1}^\infty \left(\sum_{\substack{k| n \\ k\text{ even}}} \frac{n}{k}\right)y^n \\
	\overset{(*)}{=} & G_2(y)-\sum_{n'=1}^\infty \left(\sum_{k'|n'}\frac{2n'}{2k'}\right)y^{2n'} 
	= G_2(y)-G_2(y^2).
	\end{align*}
	To see $(*)$,  note that for even $k$ and $n$ we can write $k=2k'$, $n=2n'$, and $k|n$ if and only if $k'|n'$. Then, $H_0$ and $H_1$ are just the even and odd parts of $H$:
	$$\begin{array}{rlrl}
	H_0(y) = & \frac{1}{2}\left(H(y)+H(-y)\right) & H_1(y) = & \frac{1}{2}\left(H(y)-H(-y)\right) \\
	= & \frac{1}{2}\left( G_2(y)+G_2(-y) \right)-G_2(y^2) , & = & \frac{1}{2}\left( G_2(y)-G_2(-y) \right) . \\
	\end{array}$$
	Using \cite[Lemma 2.1]{blomme2022abelian3}, they are quasi-modular forms for finite index subgroups of $SL_2(\ZZ)$.
	\end{proof}

	\begin{theo}\label{theorem-quasi-modularity}
	The generating series $F^0_{g,b}(\mu,\nu)$ and $F^1_{g,b}(\mu,\nu)$ are quasi-modular forms for some finite index subgroup of $SL_2(\ZZ)$.
	\end{theo}

	\begin{proof}
	In the section on polynomiality \ref{sec polynomiality}, we consider the floor diagrams and forget the elevators weights. Here, we instead forget the degrees of the floors. Given $b$ and  partitions $\mu$ and $\nu$ such that $\mu+\nu\vdash 2b$, there are only a finite number of genus $g$ floor diagrams up to the degree of the floors. Then, for a given $a\in\frac{1}{2}\NN$, we find the diagrams contributing to $N^\delta_{g,aE+bF}(\mu,\nu)$ by constructing partitions of $a$ satisfying such that for each \'etage $\F$, $a_\F\in\NN$, for a ground floor $\G$, $a_\G\in\frac{1}{2}\NN$ and $\delta\cdot 2a_\G\equiv\sum_{e\ni\G}w_e \text{ mod }2$,
	and such that $\sum a_\G+\sum a_\F=a$.
 
	Let $\Dfk$ be a marked floor diagram without floor degrees, and for a ground floor let $\varepsilon_\G\equiv \sum_{e\ni\G}w_e$. Given a family $(a_\F,a_\G)$ of degrees for the floors, let $\Dfk(a_\F,a_\G)$ be the corresponding diagram. By Definition \ref{definition diagram multiplicity}, its multiplicity is 
	$$m\big(\Dfk(a_\F,a_\G)\big)=W\prod_\F a_\F^{\val\F-1}\sigma_1(a_\F)\prod_\G (2a_\G)^{\val\G-1}\widetilde{\sigma_1}(2a_\G),$$
	where $W$ accounts for the contribution of the elevators weights and the $2^k$ term. The sum over all values of $a$ factors as follows:
	\begin{align*}
	 & \sum_{a}\left(\sum_{\Sigma a_\F+\Sigma a_\G=a} m\big(\Dfk(a_\F,a_\G)\big) \right) y^{2a} \\
	= & W \prod_\F\left( \sum_{a_\F=1}^\infty a_\F^{\val\F-1}\sigma_1(a_\F) y^{2a_\F}\right)\prod_\G\left( \sum_{\substack{a_\G\in\frac{1}{2}\NN \\ \delta 2a_\G\equiv \varepsilon_\G\text{ mod }2}}^\infty (2a_\G)^{\val\G-1} \widetilde{\sigma_1}(2a_\G) y^{2a_\G}\right). 
	\end{align*}
	The series in the product over the \'etages are quasi-modular forms, since they are equal to $(D^k G_2)(y^2)$ for a derivation of order $k$.  The product over the ground floors depends on  $\delta$:
 
	In $\TT M_0$, there is no parity condition on the sum, thus we recover some derivative of the generating function $H(y)$ and obtain quasi-modularity. In $\TT M_1$, depending of the value of $\varepsilon_\G$, we sum over the odd or even values, yielding $H_{\varepsilon_\G}(y)$ in any case. This also results in the quasi-modularity of the series.
	\end{proof}
	
	\begin{rem}
	For  $\TT M_0$, $\mu=\emptyset$ and $\nu=1^{2b}$, we recover the non-relative invariants of the degree $0$ cylinder, for which  quasi-modularity has already been proven in \cite{blomme2021floor} and \cite{bohm2020counts}.
	\end{rem}

\bibliographystyle{plain}
\bibliography{biblio}

\end{document}